\documentclass[a4paper, 10pt, twoside, notitlepage]{amsart}

\usepackage{amsmath,amscd}
\usepackage{amssymb}
\usepackage{amsthm}
\usepackage{comment}
\usepackage{graphicx, xcolor}

\usepackage{mathrsfs}
\usepackage[ocgcolorlinks, linkcolor=blue]{hyperref}

\usepackage{bm}
\usepackage{bbm}
\usepackage{url}

\usepackage[utf8]{inputenc}
\usepackage{mathtools,amssymb}
\usepackage{esint}
\usepackage{tikz}
\usepackage{dsfont}
\usepackage{relsize}
\usepackage{url}
\usepackage{xcolor}
\usepackage{graphicx}
\usepackage{mathrsfs}
\usepackage[shortlabels]{enumitem}
\usepackage{lineno}
\usepackage{amsmath}
\usepackage{enumitem}
\usepackage{amsthm} 
\usepackage{verbatim}
\usepackage{dsfont}
\numberwithin{equation}{section}

\allowdisplaybreaks

\mathtoolsset{showonlyrefs}

\graphicspath{{images/}}

\newtheorem{theorem}{Theorem}[section]
\newtheorem{lemma}[theorem]{Lemma}
\newtheorem{definition}[theorem]{Definition}

\newtheorem{claim}[theorem]{Claim}
\newtheorem{proposition}[theorem]{Proposition}

\newtheorem{remark}[theorem]{Remark}

\title[Reconstruction of coefficients in the double phase problem]{Reconstruction of coefficients in the double phase problem}

\author[C. I. C\^{a}rstea]{C\u{a}t\u{a}lin I. C\^{a}rstea}
\address{Department of Applied Mathematics, National Yang Ming Chiao Tung University, Hsinchu, Taiwan}
\email{catalin.carstea@gmail.com}

\author[P. Zimmermann]{Philipp Zimmermann}
\address{Departament de Matem\`atiques i Inform\`atica, Universitat de Barcelona, Barcelona, Spain}
\email{philipp.zimmermann@ub.edu}

\newcommand{\C}{{\mathbb C}}
\newcommand{\R}{{\mathbb R}}

\newcommand{\N}{{\mathbb N}}

\newcommand{\eps}{\varepsilon}

\DeclareMathOperator{\Div}{div} 
\DeclareMathOperator{\supp}{supp} 
\DeclareMathOperator{\dist}{dist} 
\DeclareMathOperator{\dom}{\Omega} 

\begin{document}

	\maketitle
	\begin{abstract}
		The main purpose of this article is to reconstruct the nonnegative coefficient $a$ in the double phase problem 
        \[
        \begin{cases}
            \Div(|\nabla u|^{p-2}\nabla u+a|\nabla u|^{q-2}\nabla u)=0&\text{ in }\Omega,\\
            u=f&\text{ on }\partial\Omega
        \end{cases}
        \]
        from the Dirichlet to Neumann (DN) map $\Lambda_a$. We show that this can be achieved, when the coefficient $a$ has H\"older continuous first order derivatives and the exponents satisfy $1<p\neq q<\infty$. Our reconstruction method relies on a careful analysis of the asymptotic behavior of the solution $u$ to the double phase problem with small or large Dirichlet datum $f$ (depending on the ordering of $p$ and $q$) as well as the related DN map $\Lambda_a$. As is common for inverse boundary value problems, we need a sufficiently rich family of special solutions to a related partial differential equation, which is independent of the coefficient one aims to reconstruct (in our case to the $p$-Laplace equation). We construct such families of solutions by a suitable linearization technique.  
		
		\medskip
		
		
		
	\end{abstract}

	\tableofcontents

	\section{Introduction}
	\label{sec: introduction}

    In the present article, we intend to uniquely recover the nonnegative (scalar) coefficient $a$ in the \emph{double phase problem}
    \begin{equation}
    \label{eq: PDE double phase}
        \begin{cases}
            \Div(|\nabla u|^{p-2}\nabla u+a|\nabla u|^{q-2}\nabla u)=0&\text{ in }\Omega,\\
            u=f&\text{ on }\partial\Omega
        \end{cases}
    \end{equation}
    from the \emph{Dirichlet to Neumann (DN) map} $\Lambda_a$. The later is formally given by the expression
    \begin{equation}
    \label{eq: formal DN map}
        \Lambda_a f\vcentcolon = \left.(|\nabla u|^{p-2}\partial_\nu u+a|\nabla u|^{q-2}\partial_\nu u)\right|_{\partial\Omega},
    \end{equation}
    where $u$ is the unique solution to \eqref{eq: PDE double phase} with boundary value $f\in W^{1,q}(\Omega)$, $\partial_\nu=\nu\cdot\nabla$ and $\nu$ denotes the outer normal unit vector field to $\partial\Omega$. The DN map $\Lambda_a$ will be rigorously introduced in Section \ref{sec: DN map double phase}. Here and throughout this article, we assume that $\Omega$ is a smoothly bounded domain in $\R^n$ with $n\geq 2$ and the exponents $p$, $q$ and coefficient $a$ satisfy the conditions
     \begin{equation} 
    \label{eq: cond on coeff}
        1<p\neq q<\infty \text{ and }0\leq a\in C^{1,\alpha}(\overline{\Omega})
    \end{equation}
    for some $0<\alpha\leq 1$. In Section \ref{sec: well-posedness}, we first show by variational methods that the Dirichlet problem \eqref{eq: PDE double phase} is uniquely solvable for boundary values in appropriate Sobolev spaces. More concretely, in the case $p<q$, it is established that for boundary values $f\in W^{1,q}(\Omega)$ the solution $u\in W^{1,p}(\Omega)$, satisfying $a|\nabla u|^q\in L^1(\Omega)$, can be constructed as the unique minimizer of the \emph{double phase functional}
    \begin{equation}
        \label{eq: Dirichlet double phase}
        \begin{cases}
            \min_{v\in M_f}\mathcal{F}(v;\Omega),\\
            \mathcal{F}(u;\Omega)\vcentcolon= \int_\Omega \left(|\nabla u|^p+\frac{p}{q}a|\nabla u|^q\right)\,dx,\\
            M_f\vcentcolon = f+W^{1,p}_0(\Omega)\text{ with }f\in W^{1,q}(\Omega).
        \end{cases}
    \end{equation}
    Clearly, a similar characterization holds in the range $p>q$. The key property of the functional $\mathcal{F}$, leading to its name, is that its growth behaviour is in general non-constant over $\Omega$ as the coefficient $a$ can vanish in some subregions of $\Omega$. The region $\{a=0\}$ is usually called the $p$ phase, the region $\{a>0\}$ the $(p,q)$ phase, and one thinks of $\partial\{a>0\}$ as modeling a phase transition. In recent years several researchers investigated the regularity properties of local minimizers of such variational integrals exhibiting different phases. For example, in the article \cite{Colombo-Mingione-bounded-minimizers-double-phase} the authors showed that bounded, local minimizers of the double phase functional are in $C^{1,\beta}(\overline{\Omega'})$, for any $\Omega'\Subset\Omega$, as long as $a\in C^{0,\alpha}(\overline{\Omega})$ and the exponents $p,q$ satisfy the additional restriction $p< q\leq p+\alpha$ (see also \cite{Baroni-Colombo-Mingione-regularity-double-phase,Filippis-Multiphase,Filippis-Multiphase-optimal} for regularity results on broader classes of double phase type functionals). For a more comprehensive discussion on the regularity theory of such variational integrals we refer the interested reader to the aforementioned articles and the references therein.

    \subsection{Main results and comparison to the literature}

    Our main result on the above mentioned inverse problem reads as follows:

      \begin{theorem}
    \label{thm: main theorem double phase}
        Let $\Omega\subset\R^n$ be a smoothly bounded domain, $1<p\neq q<\infty$ and $0<\alpha\leq 1$. If $a\in C^{1,\alpha}(\overline{\Omega})$ is a nonnegative function, then we can explicitly compute $a$ from the DN map $\Lambda_a$.
    \end{theorem}

    Note that when $p=q$ equation \eqref{eq: Dirichlet double phase} becomes the weighted $p$-Laplace equation
    \[
        \Div\left((1+a)|\nabla u|^{p-2}\nabla u\right)=0.
    \]
    The inverse problem for this equation has been investigated previously. In \cite{BrKaSa,Br,BrHaKaSa,BrIlKa,GuKaSa,KaWa,SaZh} partial results such as boundary determination of the weight, identification of inclusions, etc.~have been obtained. In dimension $n=2$, a general uniqueness result (i.e.~that $\Lambda_{a_1}=\Lambda_{a_2}$ implies that $a_1=a_2$) has been shown in \cite{CarFe3}, together with a weaker result for $n\geq3$, when the weight is real-analytic and slowly varying in one direction. The general case when $n\geq3$ remains open. Given all this, it is interesting that a reconstruction result is available for the double phase equation when $p\neq q$. To borrow an often used phrase, we can say that the $p$-Laplace nonlinearity helps with the recovery of the weighted $q$-Laplacian term.

The study of inverse problems for nonlinear equations has attracted a lot of interest over the last decade. Examples of works for semilinear equations include   \cite{FeOk,Is1,IsNa,IsSy,KrUh,KrUh2,LaLiiLinSa1,LaLiiLinSa2,Sun2} and for quasilinear equations we can refer the reader to \cite{CarFe1,CarFe2,CarFeKiKrUh,CarGhNa,CarGhUh,CarNaVa,Car,EgPiSc,HeSun,Is2,KaNa,MuUh,Sh,Sun1,Sun3,SunUh}. We want to highlight separately the articles \cite{BrKaSa,Br,BrHaKaSa,BrIlKa,CarFe3,CarKa,Car2,GuKaSa,KaWa,SaZh}, which deal with equations involving the $p$-Laplacian and related operators. Some of the results obtained in the aforementioned articles have also been extended to nonlocal operators. For example, the articles \cite{Frac-p-Lap,Frac-p-biharm} consider fractional $p$-Laplace and fractional $p$-biharmonic operators, whereas \cite{Frac-porous-medium} studies an inverse problem for nonlocal porous medium type equations.  

We would like to emphasize that the number of works for degenerate equations is still quite small. To the best of our knowledge, these questions have only been addressed with somewhat partial results in \cite{BrKaSa,Br,BrHaKaSa,BrIlKa,GuKaSa,KaWa,SaZh}, and with more general uniqueness results in the more recent papers \cite{CarFe3,CarGhNa,CarGhUh}. The results of our paper add to this newly developing field. 

The problem we consider in the present paper is somewhat similar to that considered in \cite{CarKa}, where an equation of the form 
\[
    \Div(\gamma\nabla u+a|\nabla u|^{q-2}\nabla u)=0
\]
is studied. Indeed, if $\gamma=1$, the result of \cite{CarKa} can be seen as a particular case of Theorem \ref{thm: main theorem double phase}, with $p=2$. The other paper that studies a similar equation is \cite{Car2}, where the equation 
\[
    \Div(a\nabla u+|\nabla u|^{p-2}\nabla u)=0
\]
is considered, with $a$ a (uniformly elliptic) matrix-valued coefficient. This case is partially analogous to taking $q=2$ in this paper.

Most works on inverse problems for nonlinear equations employ the so-called \emph{second/higher linearization} method, which was first introduced in \cite{Is1}. In this method one uses Dirichlet data $f$ that depend on small parameters $\epsilon_1,\ldots, \epsilon_N$, typically of the form $\sum_k\epsilon_k\phi_k$, then one considers derivatives of different orders in $\epsilon_1,\ldots,\epsilon_N$  of the equation and the DN map $\Lambda(\sum_k\epsilon_k\phi_k)$, in order to accumulate information about the coefficients appearing in the original partial differential equation (PDE). 

Here, we use a slightly different approach. For illustration purposes, suppose that we are in the $1<p<q<\infty$ regime. We take Dirichlet data of the form $f=\eps v$ and we expect that as $\eps\to0$ the $p$-Laplace term dominates the equation. Assuming $v$ is a $p$-harmonic function (which for technical reasons needs to have no critical points), we show that the corresponding solution $u_\eps$ of \eqref{eq: PDE double phase} has the asymptotic behavior 
\begin{equation}
u_\eps=\eps v+\eps^{1+q-p}R_v+o(\eps^{1+q-p}),
\end{equation}
where $R_v$ is independent of $\eps$ but depends on $v$ and $a$, as it is the unique solution in $C_0^{2,\beta}(\overline{\dom})$ (the zero indicates that it vanishes on the boundary) of a linear elliptic equation
\[
    \Div(A_v^p\nabla R)=-\Div(a|\nabla v|^{q-2}\nabla v),  
\]
with $v$-dependent matrix coefficients $A_v^p$. This asymptotic expansion of $u_\eps$ then guarantees an expansion for $\Lambda_a(\eps v)$. For any $\omega\in W^{1,q}(\Omega)$, we will show that
\begin{equation}
\lim_{\eps\to0+}\eps^{1-q}\langle(\Lambda_a-\Lambda_0)(\eps v),\omega\rangle
=\int_{\dom}(A_v^p\nabla R_v+a|\nabla v|^{q-2}\nabla v)\cdot\nabla\omega\,dx.
\end{equation}
We then make special choices of functions $\omega$ and $p$-harmonic functions $v$ in order to recover $a$.

    \subsection{Organization of the article}
    We first introduce in Section \ref{sec: notation} the main conventions imposed in this article. Afterwards, in Section \ref{sec: double phase equation} we discuss some background material on the double phase problem, like the well-posedness theory and the maximum/comparison principle. Then, in Section \ref{sec: asymptotics for small dirichlet data} we move on to the asymptotic analysis of solutions to the double phase problem with small (large) Dirichlet data. In Section \ref{sec: DN map double phase}, we give the rigorous definition of the DN map and discuss some of its properties. Section \ref{sec: p-Laplace linearization} deals with the construction of families $v_\tau$ of $p$ harmonic functions, which have a prescribed asymptotic expansion. In Section \ref{sec: proof of thm 1.1} we give the proof of Theorem \ref{thm: main theorem double phase}, which combines the material of Sections \ref{sec: asymptotics for small dirichlet data} and \ref{sec: p-Laplace linearization} with the method of linearization.

    \section{Notation}
    \label{sec: notation}

    Throughout this article, we make use of the following conventions.
    \begin{itemize}
        \item $\Omega\subset\R^n$ is a smoothly bounded domain, although all our results extend to lower regular domains (e.g.~$C^{2,\alpha}$ domains), and points in $\R^n$ are denoted by $x=(x_1,\ldots,x_n)$.
        \item For any scalar valued function $u$ on $\Omega$, we denote by $\nabla u=(\partial_1u,\ldots,\partial_n u)$ the gradient of $u$, where $\partial_j=\partial_{x_j}$ is the partial derivative with respect to $x_j$, and for vector valued functions $u$ the expression $\nabla u$ stands for its Jacobian matrix $(\nabla u)_{ij}=(\partial_j u_i)$ for $1\leq i,j\leq n$. More generally, we define the $k-$th order gradient $\nabla^k$ inductively by the formula $\nabla^k=\nabla^{k-1}\nabla $ for all $k\geq 2$.
        \item For any $k\in\N_0$, $C^k(\Omega)$ stands for the space of $k-$times continuously differentiable functions in $\Omega$ and $C^k(\overline{\Omega})$ is the subspace of $C^k(\Omega)$ such that $\nabla^{\ell} u$, $0\leq \ell\leq k$, can be continuously extended to $\overline{\Omega}$. The later space is a Banach under the norm
        \[
            \|u\|_{C^k(\overline{\Omega})}=\sum_{0\leq \ell\leq k}\|\nabla^{\ell}u\|_{L^{\infty}(\Omega)}\text{ with }\|u\|_{L^{\infty}(\Omega)}=\sup_{x\in\Omega}|u(x)|.
        \]
        We also write $u\in C^k(\overline{\Omega})$ ($u\in C^k(\Omega)$) and $\|u\|_{C^k(\overline{\Omega})}$, when $u$ is a vector valued function as the range will always be clear from the context and the particular used norm on the range does not matter for our analysis. The same applies to other function spaces, like H\"older or Sobolev spaces, and the related norms.
        \item If $0<\alpha\leq 1$ and $k\in\N_0$, then $C^{k,\alpha}(\overline{\Omega})$ is the space of $C^k(\overline{\Omega})$ functions such that the $k-$th order derivative is $\alpha$ H\"older continuous, that is, there holds
        \begin{equation}
            [\nabla^k u]_{C^{0,\alpha}(\overline{\Omega})}\vcentcolon = \sup_{x\neq y\in\Omega}\frac{|\nabla^k u(x)-\nabla^k u(y)|}{|x-y|^{\alpha}}<\infty.
        \end{equation}
        We endow $C^{k,\alpha}(\overline{\Omega})$ with the following norm
        \begin{equation}
            \|u\|_{C^{k,\alpha}(\overline{\Omega})}=\|u\|_{C^k(\overline{\Omega})}+[\nabla^k u]_{C^{0,\alpha}(\overline{\Omega})},
        \end{equation}
        under which they become Banach spaces.
        \item For all $1\leq p\leq \infty$ and $k\in\N_0$, we denote by $L^p(\Omega)$ the space of measurable functions $u$ in $\Omega$ such that
        \[
            \|u\|_{L^p(\Omega)}\vcentcolon = \left(\int_\Omega |u|^p\,dx\right)^{1/p}<\infty
        \]
        and by $W^{k,p}(\Omega)$ the space of $u\in L^p(\Omega)$ such that the distributional derivatives $\nabla^{\ell} u$, $0\leq \ell\leq k$, belong to $L^p(\Omega)$. We endow the Sobolev spaces $W^{k,p}(\Omega)$ with the norm
        \[
            \|u\|_{W^{k,p}(\Omega)}=\left(\sum_{0\leq\ell\leq k}\|\nabla^\ell u\|_{L^p(\Omega)}^p\right)^{1/p}.
        \]
        As usual, the spaces $L^p_{loc}(\Omega)$ and $W^{k,p}_{loc}(\Omega)$ consist of all measurable functions $u$ such that $u\in L^p(\Omega')$ and $u\in W^{k,p}(\Omega')$, respectively, for all $\Omega'\Subset \Omega$.
    \end{itemize}

    \section{The double phase problem}
    \label{sec: double phase equation}

    Here, we collect some basic results on the double phase problem: In Section \ref{sec: well-posedness} we show the well-posedness of \eqref{eq: PDE double phase}, in Section \ref{sec: local minimizer} we briefly explain its relation to local minimizers of the double phase functional and establish in Section \ref{sec: maximum principle} a maximum/comparison principle.

    \subsection{Well-posedness}
    \label{sec: well-posedness}

     Let us start by introducing for any bounded Lipschitz domain $\Omega\subset\R^n$, $a\in L^{\infty}(\Omega)$ with $a\geq 0$ and $1<p\leq q<\infty$, the Sobolev type spaces
    \begin{equation}
    \label{eq: solution space}
    \begin{split}
        W^{1,p,q}_a(\Omega)&\vcentcolon = \{u\in W^{1,p}(\Omega)\,;\,a|\nabla u|^q\in L^1(\Omega)\},\\
        W^{1,p,q}_{a,loc}(\Omega)&\vcentcolon =\{u\in W^{1,p}_{loc}(\Omega)\,;\,a|\nabla u|^q\in L^1_{loc}(\Omega)\},\\
        W^{1,p,q}_{a,0}(\Omega)&\vcentcolon =W^{1,p}_0(\Omega)\cap W^{1,p,q}_a(\Omega).
    \end{split}
    \end{equation}
    Under the norm
    \begin{equation}
    \label{eq: norm on W 1pq}
        \|u\|_{W^{1,p,q}_a(\Omega)}\vcentcolon = \|u\|_{W^{1,p}(\Omega)}+\|\nabla u\|_{L^q(\Omega;a)},
    \end{equation}
    where 
    \[
        \|v\|_{L^q(\Omega;a)}\vcentcolon =\left(\int_\Omega a|v|^q\,dx\right)^{1/q},
    \]
    the spaces $W^{1,p,q}_a(\Omega)$ and $W^{1,p,q}_{a,0}(\Omega)$ defined in \eqref{eq: solution space} become Banach spaces. This is a direct consequence of the fact that $W^{1,p}(\Omega)$ and $L^q(\Omega;a)$ are Banach spaces.
    We have the following elementary lemma.
  
  \begin{lemma}
    \label{lemma: properties of W 1pq}
        Let $\Omega\subset\R^n$ be a bounded Lipschitz domain, $1<p\leq q<\infty$ and assume that $a\in L^{\infty}(\Omega)$ with $a\geq 0$. The space $W^{1,p,q}_a(\Omega)$ endowed with the norm $\|\cdot\|_{W^{1,p,q}_a(\Omega)}$ is a separable, reflexive Banach space.
    \end{lemma}
    \begin{proof}      
        \textbf{Step 1:} First, we show that $W^{1,p,q}_a(\Omega)$ is a Banach space. From the definitions it is clear that $W^{1,p,q}_a(\Omega)$ is a normed space and hence we only need to show that its complete. For this purpose, assume that $(u_n)_{n\in\N}\subset W^{1,p,q}_a(\Omega)$ is a Cauchy sequence. Then $(u_n)_{n\in\N}$ is a Cauchy sequence in $W^{1,p}(\Omega)$ and $(\nabla u_n)_{n\in\N}$ in $L^q(\Omega;a)$. By completeness of these spaces, there exist $u\in W^{1,p}(\Omega)$ and $v\in L^q(\Omega;a)$ such that
          \[
            u_n\to u\text{ in }W^{1,p}(\Omega)\text{ and }\nabla u_n\to v\text{ in }L^q(\Omega;a)
        \]
        as $n\to\infty$. As $L^r$ convergence implies a.e. convergence for a suitable subsequence, we may infer $v=\nabla u$ a.e. with respect to the measure $a\,dx$ and hence $u\in W^{1,p,q}_a(\Omega)$ such that $u_n\to u$ in $W^{1,p,q}_a(\Omega)$. This proves the completeness of $W^{1,p,q}_a(\Omega)$.\\

        \noindent\textbf{Step 2.} Next, we prove that $W^{1,p,q}_a(\Omega)$ is seperable and reflexive. To this end, let us consider the mapping
        \begin{equation}
        \label{eq: map j}
            j\colon W^{1,p,q}_a(\Omega)\to L^p(\Omega)\times L^p(\Omega)\times L^q(\Omega)\text{ with }j(u)\vcentcolon =  (u,\nabla u, a^{1/q}\nabla u).
        \end{equation}
        Let $\|\cdot\|_{W^{1,p,q}_a(\Omega)}^\ast$ be the equivalent norm
        \[
            \|u\|_{W^{1,p,q}_a(\Omega)}^\ast\vcentcolon = \|u\|_{L^p(\Omega)}+\|\nabla u\|_{L^p(\Omega)}+\|a^{1/q}\nabla u\|_{L^q(\Omega)}
        \]
        on $W^{1,p,q}_a(\Omega)$. Then the map $j$, defined by \eqref{eq: map j}, is an isometric embedding, when we use the product norm
        \[
            \|(u,v,w)\|^{\ast}_{L^p(\Omega)\times L^p(\Omega)\times L^q(\Omega)}=\|u\|_{L^p(\Omega)}+\|v\|_{L^p(\Omega)}+\|w\|_{L^q(\Omega)}
        \]
        for $(u,v,w)\in L^p(\Omega)\times L^p(\Omega)\times L^q(\Omega)$. As the target space is a cartesian product of separable spaces, it is separable and $j(W^{1,p,q}_a(\Omega))$ is separable. Now, using that $j$ is an isometry, we see that $W^{1,p,q}_a(\Omega)$ is separable. The space $W^{1,p,q}_a(\Omega)$ is also separable under the original norm as separability is preserved in the class of equivalent norms.
        
        This time we endow $W^{1,p,q}_a(\Omega)$ and $L^p(\Omega)\times L^p(\Omega)\times L^q(\Omega)$ with the equivalent norms
        \[
        \begin{split}
            &\|u\|'_{W^{1,p,q}_a(\Omega)} \vcentcolon = (\|u\|^r_{L^p(\Omega)}+\|\nabla u\|^r_{L^p(\Omega)}+\|a^{1/q}\nabla u\|^r_{L^q(\Omega)})^{1/r},\\
            &\|(u,v,w)\|'_{L^p(\Omega)\times L^p(\Omega)\times L^q(\Omega)} \vcentcolon = (\|u\|_{L^p(\Omega)}^r+\|v\|_{L^p(\Omega)}^r+\|w\|_{L^q(\Omega)}^r)^{1/r}
        \end{split}
        \]
        for some $1<r<\infty$. Then $j$ is again an isometry. From \cite[Theorem 1]{clarkson1936uniformly} and the fact that Clarkson's inequalities imply the uniform convexity of $L^s(\Omega)$ for $1<s<\infty$, it follows that $L^p(\Omega)\times L^p(\Omega)\times L^q(\Omega)$ is uniformly convex. Next, observe that by Step 1 the linear subspace $j(W^{1,p,q}_a(\Omega))$ of $L^p(\Omega)\times L^p(\Omega)\times L^q(\Omega)$ is closed. Hence, $j(W^{1,p,q}_a(\Omega))$ is uniformly convex. The Milmann--Pettis theorem now shows that $j(W^{1,p,q}_a(\Omega))$ is reflexive. As reflexivity is preserved under isomorphisms, we can conclude that $W^{1,p,q}_a(\Omega)$ is reflexive. The same remains true for the usual norm on $W^{1,p,q}_a(\Omega)$.
    \end{proof} 
    
    To prove the well-posedness of \eqref{eq: PDE double phase}, we need in addition the following auxiliary lemma.

    \begin{lemma}[{cf.~\cite[eq.~(2.2)]{Simon}, \cite[Lemma 5.1-5.2]{Estimate_p_Laplacian} and \cite[Appendix~A]{SaZh}}]
    \label{auxialiary lemma}
        For all $1<r<\infty$, $n\in\N$ and $x,y\in\R^n$, there holds
        \begin{equation}
            \label{eq: minimality}
            |x|^r\geq |y|^r+r|y|^{r-2}y\cdot(x-y),
        \end{equation}
		\begin{equation}
			\label{eq: monotonicity est}
			\left( |x|^{r-2}x-|y|^{r-2}y\right) \cdot (x-y)\gtrsim \begin{cases}
			    |x-y|^r, &\text{ for }r\geq 2,\\
                \frac{|x-y|^2}{(|x|+|y|)^{2-r}},&\text{ for }1<r<2,
			\end{cases}
		\end{equation}
		\begin{equation}
			\label{eq: estimate mikko}	
			\begin{split}	
				\left| |x|^{r-2}x-|y|^{r-2}y\right|\lesssim &(|x|+|y|)^{r-2}|x-y|
			\end{split}
		\end{equation}
        and 
        \begin{equation}
        \label{eq: difference rth power}
            ||x|^r-|y|^r|\leq r(|x|^{r-1}+|y|^{r-1})(x-y).
        \end{equation}
    \end{lemma}
    Using the above spaces and the previous auxiliary lemma, we can establish the following well-posedness result:
    \begin{theorem}[Well-posedness]
    \label{thm: well-posedness double phase}
        Let $\Omega\subset\R^n$ be a bounded Lipschitz domain, $1<p\leq q<\infty$ and $a\in L^{\infty}(\Omega)$ a nonnegative function. Then the Dirichlet problem \eqref{eq: PDE double phase} has for any $f\in W^{1,q}(\Omega)$ a unique solution $u\in W^{1,p,q}_a(\Omega)$, that is there holds
        \begin{equation}
        \label{eq: weak solutions double phase}
            \int_\Omega \left(|\nabla u|^{p-2}\nabla u+a|\nabla u|^{q-2}\nabla u\right)\cdot\nabla \varphi\,dx=0
        \end{equation}
        for all $\varphi\in W^{1,p,q}_{a,0}(\Omega)$ and $u=f$ on $\partial\Omega$ in the sense of traces. 
        
        Moreover, $u$ can be characterized as the unique minimizer of \eqref{eq: Dirichlet double phase} and satisfies the estimates
        \begin{equation}
        \label{eq: energy estimate W 1p}
            \|u\|_{W^{1,p}(\Omega)}\leq C(\|\nabla f\|_{L^q(\Omega)}+\|\nabla f\|_{L^q(\Omega)}^{q/p}+\|f\|_{L^p(\partial\Omega)})
        \end{equation}
        and
    \begin{equation}
    \label{eq: energy estimate}
        \|u\|_{W^{1,p,q}_a(\Omega)}\leq C(\|\nabla f\|_{L^q(\Omega)}+\|\nabla f\|_{L^q(\Omega)}^{p/q}+\|\nabla f\|_{L^q(\Omega)}^{q/p}+\|f\|_{L^q(\partial\Omega)}).
    \end{equation}
    for some $C>0$ depending increasingly on $\|a\|_{L^{\infty}(\Omega)}$.
    \end{theorem}

    \begin{proof}
        First, we consider the minimization problem \eqref{eq: Dirichlet double phase}. Since $f\in W^{1,q}(\Omega)\subset W^{1,p}(\Omega)$ and $a\in L^{\infty}(\Omega)$, we have $\mathcal{F}(f;\Omega)<\infty$. Therefore, the strict convexity of the integrand of \eqref{eq: Dirichlet double phase} and \cite[Theorem 3.30]{dacorogna2007direct} ensure the existence of a unique minimizer $u\in M_f\cap W^{1,p,q}_a(\Omega)$ of \eqref{eq: Dirichlet double phase}. By the minimality property of $u$, we have
        \[
            \|\nabla u\|_{L^p(\Omega)}^p\leq \mathcal{F}(u;\Omega)\leq \mathcal{F}(f;\Omega).
        \]
        Using the Poincar\'e inequality
        \begin{equation}
        \label{eq: Poincare}
            \int_\Omega |v|^p\,dx\leq C\left(\int_\Omega |\nabla v|^p\,dx+\int_{\partial\Omega}|v|_{\partial\Omega}|^p\,d\sigma\right)
        \end{equation}
        for $v\in W^{1,p}(\Omega)$ and some $C>0$ only depending on $n,p$ and $\Omega$ (see, for example, \cite{Elliptic-PDEs-Xavi}), we get
        \begin{equation}
        \label{eq: W 1p control}
            \|u\|_{W^{1,p}(\Omega)}\leq C((\mathcal{F}(f;\Omega))^{1/p}+\|f\|_{L^p(\partial\Omega)}).
        \end{equation}
        Furthermore, by the minimality property we also have
        \begin{equation}
        \label{eq: Lqa control}
            \|\nabla u\|_{L^q(\Omega;a)}^q\lesssim \mathcal{F}(u;\Omega)\lesssim \mathcal{F}(f;\Omega).
        \end{equation}
        Using $p\leq q$ and $f\in W^{1,q}(\Omega)$, we can estimate
        \[
            \mathcal{F}(f;\Omega)\leq C(1+\|a\|_{L^{\infty}(\Omega)})(\|\nabla f\|_{L^q(\Omega)}^p+\|\nabla f\|_{L^q(\Omega)}^q)
        \]
        and thus
        \begin{equation}
        \label{eq: energy f}
            \begin{split}
                (\mathcal{F}(f;\Omega))^{1/p}&\leq C(1+\|a\|_{L^{\infty}(\Omega)})^{1/p}(\|\nabla f\|_{L^q(\Omega)}+\|\nabla f\|_{L^q(\Omega)}^{q/p})\\
                (\mathcal{F}(f;\Omega))^{1/q}&\leq C(1+\|a\|_{L^{\infty}(\Omega)})^{1/q}(\|\nabla f\|_{L^q(\Omega)}+\|\nabla f\|_{L^q(\Omega)}^{p/q}).
            \end{split}
        \end{equation}
        Hence, from  \eqref{eq: W 1p control}, \eqref{eq: Lqa control}, \eqref{eq: energy f}, \eqref{eq: norm on W 1pq} and $p\leq q$, we get \eqref{eq: energy estimate W 1p} and
        \[
        \begin{split}
            &\|u\|_{W^{1,p,q}_a(\Omega)}\leq C((\mathcal{F}(f;\Omega))^{1/p}+\|f\|_{L^p(\partial\Omega)}+(\mathcal{F}(f;\Omega))^{1/q})\\
            &\leq C(1+\|a\|_{L^{\infty}(\Omega)})^{1/p}(\|\nabla f\|_{L^q(\Omega)}+\|\nabla f\|_{L^q(\Omega)}^{p/q}+\|\nabla f\|_{L^q(\Omega)}^{q/p}+\|f\|_{L^q(\partial\Omega)}).
        \end{split}
        \]
        This establishes \eqref{eq: energy estimate}.

        Next, we show that $u$ is in fact a solution of \eqref{eq: PDE double phase}. By construction of $u$, we know that $u+\eps\varphi=(u-f)+(f+\eps\varphi)\in M_f\cap W^{1,p,q}_a(\Omega)$ for any $\eps>0$ and $\varphi\in W^{1,p,q}_{a,0}(\Omega)$. Thus, the minimality of $u$ implies that $\eps\mapsto \mathcal{F}(u+\eps\varphi;\Omega)$ attains its minimium at $\eps=0$. Using $u,\varphi\in W^{1,p,q}_a(\Omega)$ and the dominated convergence theorem, we may calculate
        \[
        \begin{split}
            0&=\left.\frac{d}{d\eps}\right|_{\eps=0}\mathcal{F}(u+\eps\varphi;\Omega)\\
            &=\left.\frac{d}{d\eps}\right|_{\eps=0}\int_\Omega \left(|\nabla u+\eps\nabla \varphi|^p+\frac{p}{q}a|\nabla u+\eps\nabla \varphi|^q\right)\,dx\\
            &=p\int_\Omega \left(|\nabla u|^{p-2}\nabla u+a|\nabla u|^{q-2}\nabla u\right)\cdot\nabla \varphi\,dx.
        \end{split}
        \]
        Therefore, $u$ is a weak solution of \eqref{eq: PDE double phase} in the sense of \eqref{eq: weak solutions double phase}. 

        Finally, it remains to show that the solution $u$ is unique. This in turn is a consequence of the observation that any solution $w\in W^{1,p,q}_a(\Omega)$ of \eqref{eq: PDE double phase} is a minimizer of $\mathcal{F}(\cdot,\Omega)$ over $M_f$. In fact, if $\Bar{u}\in W^{1,p,q}_a(\Omega)$ is another solution of \eqref{eq: PDE double phase}, then $\Bar{u}$ also minimizes $\mathcal{F}(\cdot;\Omega)$ over $M_f$ and as minimizers of this problem are unique, we can conclude that $\Bar{u}=u$.

        For proving that solutions are minimizers, let us choose any $v\in M_f\cap W^{1,p,q}_a(\Omega)$. Then, using \eqref{eq: minimality} and \eqref{eq: weak solutions double phase}, we deduce that
        \[
        \begin{split}
            \mathcal{F}(v;\Omega)&=\int_\Omega \left(|\nabla v|^p+\frac{p}{q}a|\nabla v|^q\right)\,dx\\
            &\geq \int_\Omega \left(|\nabla w|^p+\frac{p}{q}a|\nabla w|^q\right)\,dx\\
            &\quad+p\int_\Omega \left(|\nabla w|^{p-2}\nabla w+a|\nabla w|^{q-2}\nabla w\right)\cdot\nabla (v-w)\,dx\\
            &=\mathcal{F}(w;\Omega),
        \end{split}
        \]
        where we used that $v-w\in W^{1,p,q}_{a,0}(\Omega)$. On the other hand, if $v\in M_f$ does not belong to $W^{1,p,q}_a(\Omega)$, then $\mathcal{F}(v;\Omega)=\infty$ and the previous estimate holds trivially. Hence, we have shown that any solution $w\in W^{1,p,q}_a(\Omega)$ of \eqref{eq: PDE double phase} is a minimizer of \eqref{eq: Dirichlet double phase}. 

        Hence, in summary, we have demonstrated that the unique minimizer $u$ of \eqref{eq: Dirichlet double phase} is the unique $W^{1,p,q}_a(\Omega)$-solution of \eqref{eq: PDE double phase} and moreover it satisfies the estimate \eqref{eq: energy estimate}.
    \end{proof}

    \subsection{Local minimizers of the double phase functional}
    \label{sec: local minimizer}
    
    The only purpose of this section is to show that solutions to the Dirichlet problem \eqref{eq: PDE double phase} are local minimizers of the double phase functional in the sense of Definition \ref{def: local minimizers}. By combining this with the maximum principle (Theorem \ref{thm: max principle}), one then sees that all (local) regularity results established in the work \cite{Colombo-Mingione-bounded-minimizers-double-phase} can be applied to solutions of \eqref{eq: PDE double phase}, when the Dirichlet datum $f$ is bounded and in $W^{1,q}(\Omega)$.

    \begin{definition}
    \label{def: local minimizers}
            Let $\Omega\subset\R^n$ be a bounded Lipschitz domain, $1<p\leq q<\infty$, $0<\alpha\leq 1$ and $a\in C^{0,\alpha}(\overline{\Omega})$ a nonnegative function. Then, a function $u\in W^{1,1}_{loc}(\Omega)$ is said to be a \emph{local minimizer} of $\mathcal{F}(\cdot;\Omega)$, when $\mathcal{F}(u;\Omega')<\infty$ for all $\Omega'\Subset\Omega$ and there holds
            \begin{equation}
            \label{eq: local minimizers}
                \mathcal{F}(u;\supp(u-v))\leq \mathcal{F}(v;\supp(u-v))
            \end{equation}
            for all $v\in W^{1,1}_{loc}(\Omega)$ such that $\supp(u-v)\subset\Omega$.
    \end{definition}

    \begin{remark}
    \label{remark: support}
        Note that this definition makes sense as $\supp(u-v)\subset\Omega$ is contained in some $\Omega'\Subset\Omega$. In fact, as $\supp(u-v)\subset\Omega$ and $\partial\Omega$ are closed and compact sets, respectively, we have $\dist(\supp(u-v),\partial\Omega)\geq \delta$ for some $\delta>0$ and so one can take $\Omega'=\{x\in\Omega\,;\,\dist(x,\partial\Omega)> \delta/2\}$. Furthermore, let us note that by \eqref{eq: solution space} the first part of Definition \ref{def: local minimizers} means nothing else than $u\in W^{1,p,q}_{a,loc}(\Omega)$.
    \end{remark}

    \begin{lemma}
    \label{lemma: solutions are local minimizers}
         Let $\Omega\subset\R^n$ be a bounded Lipschitz domain, $1<p\leq q<\infty$, $0<\alpha\leq 1$ and $a\in C^{0,\alpha}(\overline{\Omega})$ a nonnegative function. If $u\in W^{1,p,q}_a(\Omega)$ is a minimizer of \eqref{eq: Dirichlet double phase} for some $f\in W^{1,q}(\Omega)$, then it is also a local minimizer in the sense of Definition \ref{def: local minimizers}.
    \end{lemma}

    \begin{proof}
        Let $v\in W^{1,1}_{loc}(\Omega)$ and suppose that $\supp(u-v)\subset\Omega$. Assume first that $v\notin W^{1,p}(\Omega)$, then $\mathcal{F}(v;\Omega)=\infty$ and we have trivially
        \begin{equation}
        \label{eq: minimizer prop}
            \mathcal{F} (u;\Omega)\leq \mathcal{F}(v;\Omega).
        \end{equation}
        Next, we assert that \eqref{eq: minimizer prop} also holds when $v\in W^{1,p}(\Omega)$. In this case, we write $v=u+\varphi$, where $\varphi=v-u\in W^{1,p}(\Omega)$. Moreover, by assumption and Remark \ref{remark: support} we know that $\supp\varphi\subset\Omega'$ for some $\Omega'\Subset \Omega$. Let us denote by $(\rho_\eps)_{\eps>0}\subset C_c^{\infty}(\R^n)$ the standard mollifier and define the functions $\varphi_\eps =\rho_\eps\ast \varphi\in C_c^{\infty}(\Omega)$ for some sufficiently small $\eps>0$. By the properties of mollification, we know that
        \[
            \nabla\varphi_\eps\to \nabla\varphi\text{ for a.e. }x\in\Omega
        \]
        along a suitable subsequence. Moreover, by Young's inequality there holds
        \[
            |\nabla \varphi_\eps(x)|\leq \|\rho_\eps\|_{L^{p'}(\Omega)}\|\nabla\varphi\|_{L^p(\Omega)}.
        \]
        So, by boundedness of $\Omega$, the dominated convergence theorem and the minimality of $u\in M_f$, we get
        \begin{equation}
        \label{eq: almost local min estimate}
            \mathcal{F}(u;\Omega)\leq \limsup_{\eps\to 0}\mathcal{F}(u+\varphi_\eps;\Omega)=\mathcal{F}(v;\Omega).
        \end{equation}
        Hence, also in this case we have \eqref{eq: minimizer prop}. Finally, we may observe
        \[
        \begin{split}
            \mathcal{F}(v;\Omega)&=\mathcal{F}(v;\supp(u-v))+\mathcal{F}(v;\R^n\setminus \supp(u-v))\\
            &=\mathcal{F}(v;\supp(u-v))+\mathcal{F}(u;\R^n\setminus \supp(u-v)).
        \end{split}
        \]
        Inserting this into the right hand side of \eqref{eq: almost local min estimate} and using the same decomposition for the left hand side, we achieve the identity \eqref{eq: local minimizers}. This concludes the proof.
    \end{proof}

  \subsection{Maximum and comparison principle}
  \label{sec: maximum principle}

    Next, let us recall that solutions to the double phase problem \eqref{eq: PDE double phase} satisfy a maximum principle (see~\cite[Lemma 2.2]{Colombo-Mingione-bounded-minimizers-double-phase} and \cite{leonetti1991maximum}):

    \begin{theorem}[{Maximum principle}]
    \label{thm: max principle}
        Let $\Omega\subset\R^n$ be a bounded Lipschitz domain, $1<p\leq q<\infty$ and $a\in L^{\infty}(\Omega)$ a nonnegative function. If $u\in W^{1,p,q}_a(\Omega)$ solves the Dirichlet problem \eqref{eq: PDE double phase} with $f\in L^{\infty}(\Omega)\cap W^{1,q}(\Omega)$, then
        \begin{equation}
        \label{eq: estimate max principle}
            \|u\|_{L^{\infty}(\Omega)}\leq \|f\|_{L^{\infty}(\Omega)}.
        \end{equation}
    \end{theorem}

   On the one hand, the advantage of this formulation of the maximum principle is that it is also valid in the vector-valued case $N\geq 2$, in the sense that there holds
        \begin{equation}
        \label{eq: max principle vector}
            \|u\|_{L^{\infty}(\Omega)}\leq N^{1/2}\|f\|_{L^{\infty}(\Omega)}.
        \end{equation}
    On the other hand, Theorem \ref{thm: max principle} can be derived in the scalar case by very elementary considerations, namely from a weak comparison principle. As a preparatory step, we demonstrate the following weak maximum principle.

    \begin{lemma}[Weak maximum principle]
    \label{lemma: weal maximum principle}
        Let $\Omega\subset\R^n$ be a bounded Lipschitz domain, $1<p\leq q<\infty$ and $a\in L^{\infty}(\Omega)$ a nonnegative function. If $u\in W^{1,p,q}_a(\Omega)$ satisfies
        \begin{equation}
             \label{eq: PDE double phase weak max}
        \begin{cases}
            -\Div(|\nabla u|^{p-2}\nabla u+a|\nabla u|^{q-2}\nabla u)\geq 0&\text{ in }\Omega,\\
            u\geq 0&\text{ on }\partial\Omega,
        \end{cases}
        \end{equation}
        then $u\geq 0$ in $\Omega$.
    \end{lemma}

    The proof is very standard, but for the readers convenience and later usage we present the short argument.

    \begin{proof}
        Let us write $u=u_+- u_{-}$, where $u_{\pm}=\max(\pm u,0)$. It is well-known that $u\in W^{1,p,q}_a(\Omega)$ implies $u_{\pm}\in W^{1,p,q}_a(\Omega)$ with
        \begin{equation}
        \label{eq: gradient of pos neg part}
            \nabla u_{+}=\chi_{\{u>0\} }\nabla u\text{ and }\nabla u_{-}=-\chi_{\{u< 0\}}\nabla u.
        \end{equation}
        Furthermore, $u\geq 0$ on $\partial\Omega$ implies $u_{-}\in W^{1,p,q}_{a,0}(\Omega)$. Also, recall that $u$ is a solution to \eqref{eq: PDE double phase weak max} means nothing else than
         \begin{equation}
             \label{eq: sol of PDE double phase weak max}
            \int_\Omega \left(|\nabla u|^{p-2}\nabla u+a|\nabla u|^{q-2}\nabla u\right)\cdot\nabla \varphi\,dx\geq 0
        \end{equation}
        for all $\varphi\in W^{1,p,q}_{0}(\Omega)$ with $\varphi\geq 0$. In particular, we can take $\varphi=u_{-}$ to get
        \[
            \begin{split}
                0&\leq \int_\Omega \left(|\nabla u|^{p-2}\nabla u+a|\nabla u|^{q-2}\nabla u\right)\cdot\nabla u_{-}\,dx\\
                &\overset{\eqref{eq: gradient of pos neg part}}{=}-\int_\Omega \left(|\nabla u_{-}|^{p}+a|\nabla u_{-}|^{q}\right)\,dx.
            \end{split}
        \]
        By the Poincar\'e inequality \eqref{eq: Poincare}, we get $u_{-}=0$ in $\Omega$ and so the claim follows.
    \end{proof}

    The above technique of proof can be used to prove a comparison principle, which in turn implies the maximum principle as stated in Theorem \ref{thm: max principle} (use $\pm \|f\|_{L^{\infty}(\Omega)}$ as comparing functions). 

    \begin{proposition}[Comparison principle]
    \label{prop: comparison}
        Let $\Omega\subset\R^n$ be a bounded Lipschitz domain, $1<p\leq q<\infty$ and $a\in L^{\infty}(\Omega)$ a nonnegative function. If $u,v\in W^{1,p,q}_a(\Omega)$ solve \eqref{eq: PDE double phase} with boundary values $u_0,v_0\in W^{1,q}(\Omega)$ such that $u_0\geq v_0$ on $\partial\Omega$.
        Then there holds $u\geq v$ in $\Omega$.
    \end{proposition}
    \begin{proof}
        Similarly as in the proof of Lemma \ref{lemma: weal maximum principle}, we may test the PDEs for $u$ and $v$, respectively, by $\varphi=(u-v)_{-}\in W^{1,p,q}_{a,0}(\Omega)$. After subtracting the resulting identities and applying Lemma \ref{auxialiary lemma}, we get
        \begin{equation}
        \label{eq: first identity comparison}
         \begin{split}
             0&=-\int_\Omega \left( |\nabla u|^{p-2}\nabla u-|\nabla v|^{p-2}\nabla v\right) \cdot \nabla (u- v)_{-}\,dx\\
            &\quad -\int_\Omega a\left(|\nabla u|^{q-2}\nabla u-|\nabla v|^{q-2}\nabla v\right)\cdot\nabla (u-v)_{-}\,dx\\
            &=\int_\Omega \chi_{\{u<v\}}\left( |\nabla u|^{p-2}\nabla u-|\nabla v|^{p-2}\nabla v\right) \cdot \nabla (u- v)\,dx\\
            &\quad +\int_\Omega\chi_{\{u<v\}} a\left(|\nabla u|^{q-2}\nabla u-|\nabla v|^{q-2}\nabla v\right)\cdot\nabla (u-v)\,dx\\
            &\geq \int_\Omega \chi_{\{u<v\}}\left( |\nabla u|^{p-2}\nabla u-|\nabla v|^{p-2}\nabla v\right) \cdot \nabla (u- v)\,dx.
         \end{split}
        \end{equation}
        For $p\geq 2$, we deduce from formula \eqref{eq: monotonicity est} the estimate
        \[
            0\gtrsim \int_\Omega \chi_{\{u<v\}}|\nabla u-\nabla v|^p\,dx=\int_\Omega|\nabla (u-v)_{-}|^p\,dx.
        \]
         Hence, from the Poincar\'e inequality and $(u-v)_{-}=(u_0-v_0)_{-}=0$ on $\partial\Omega$, we get $(u-v)_{-}=0$ in $\Omega$, which is equivalent to $u\geq v$ in $\Omega$.

         For $p\leq 2$, we may observe that \eqref{eq: monotonicity est} implies
         \[
            |x-y|^p\lesssim [\left( |x|^{p-2}x-|y|^{p-2}y\right) \cdot (x-y)]^{p/2}(|x|+|y|)^{(2-p)p/2}
        \]
        for all $x,y\in\R^n$. Using H\"older's inequality with $\frac{2-p}{2}+\frac{p}{2}=1$, we have
        \[
        \begin{split}
            &\int_\Omega |\nabla (u-v)_{-}|^p\,dx=\int_\Omega \chi_{\{u<v\}}|\nabla (u-v)|^p\,dx\\
            &\lesssim \int_\Omega \chi_{\{u<v\}}[\left( |\nabla u|^{p-2}\nabla u-|\nabla v|^{p-2}\nabla v\right) \cdot \nabla (u-v)](|\nabla u|+|\nabla v|)^{(2-p)p/2}\,dx \\
            &\lesssim \||\nabla u|+|\nabla v|\|_{L^p(\Omega)}^{\frac{p(2-p)}{2}}\left(\int_\Omega \chi_{\{u<v\}} \left( |\nabla u|^{p-2}\nabla u-|\nabla v|^{p-2}\nabla v\right) \cdot (\nabla (u- v))\,dx\right)^{p/2}\\
            &\lesssim \||\nabla u|+|\nabla v|\|_{L^p(\Omega)}^{\frac{p(2-p)}{2}}\left(\int_\Omega \chi_{\{u<v\}} \left( |\nabla u|^{p-2}\nabla u-|\nabla v|^{p-2}\nabla v\right) \cdot \nabla (u-v)\,dx\right.\\
            &\quad \left.+\int_\Omega \chi_{\{u<v\}} a\left(|\nabla u|^{q-2}\nabla u-|\nabla v|^{q-2}\nabla v\right)\cdot\nabla (u-v)\,dx\right)^{p/2}.
        \end{split}
        \]
        By \eqref{eq: first identity comparison} the expression in the last line is identically zero and so we can again conclude from Poincar\'e's inequality and $u_0\geq v_0$ on $\partial\Omega$ that $u\geq v$ in $\Omega$. This concludes the proof.
    \end{proof}

\section{Asymptotic expansion of solutions to the double phase}
\label{sec: asymptotics for small dirichlet data}

In this section, we consider the asymptotic behaviour of solutions to the double phase problem \eqref{eq: PDE double phase}, when the boundary values $f$ become very small ($p<q$) or very large ($p>q$).

\subsection{Asymptotic expansion of solutions for \texorpdfstring{$p<q$}{p smaller q}}
\label{subsec: Asymptotic expansion of solutions to the double phase equation with small Dirichlet datum}

Here, we prove that for any smooth $p$-harmonic function $v$ without critical points, the solution $u_\eps$ to 
    \begin{equation}
    \label{eq: epsilon PDE}
    \begin{cases}
            \Div(|\nabla u|^{p-2}\nabla u+a|\nabla u|^{q-2}\nabla u)=0&\text{ in }\Omega,\\
            u=\eps v&\text{ on }\partial\Omega,
        \end{cases}
    \end{equation}
    provided by Theorem \ref{thm: well-posedness double phase}, can be expanded as
    \begin{equation}
    \label{eq: asymptotics u eps}
        u_\eps=\eps v+\eps^{1+q-p}R_v+o(\eps^{1+q-p})
    \end{equation}
    for a suitable function $R_v$ and sufficiently small $\eps>0$. Before establishing the asymptotic expansion \eqref{eq: asymptotics u eps}, we need to demonstrate some preliminary H\"older estimates in $C^{0,\beta}(\overline{\Omega})$ and to shorten the notation we will sometimes omit the dependence on the domain $\Omega$ in the appearing norms and write, for example, $\|\cdot\|_{C^{0,\beta}}$ and $\|\cdot\|_{L^{\infty}}$ instead of $\|\cdot\|_{C^{0,\beta}(\overline{\Omega})}$ and $\|\cdot\|_{L^{\infty}(\Omega)}$ . 

    \begin{lemma}
    \label{auxiliary lemma 2}
        Let $\Omega\subset\R^n$ be a smoothly bounded domain, $0<\beta<1$ and $s\in\R\setminus \{0\}$. Moreover, assume that $u\in C^{1,\beta}(\overline{\Omega})$ satisfies
        \begin{equation}
        \label{eq: lower bound on gradient of u}
            \inf_{x\in\overline{\Omega}}|\nabla u(x)|\geq \lambda_1
        \end{equation}
        for some $\lambda_1>0$. 
        \begin{enumerate}[(i)]
            \item\label{large s} If $s\geq 1$, then there holds
            \begin{equation}
            \label{eq: estimate large s}
                \||\nabla u|^s\|_{C^{0,\beta}}\leq \|\nabla u\|_{L^{\infty}}^s+2s\|\nabla u\|_{C^{0,\beta}}^s.
            \end{equation}
            \item\label{small s} If $0<s<1$, then there holds
            \begin{equation}
            \label{eq: estimate small s}
                \||\nabla u|^s\|_{C^{0,\beta}}\leq \|\nabla u\|_{L^{\infty}}^s+\frac{2s}{\prod_{j=1}^{k_s}\lambda_1^{2^{j-1}s}}\|\nabla u\|_{C^{0,\beta}}^{2^{k_s}s}
            \end{equation}
            where $k_s\in\N$ is given by
            \begin{equation}
            \label{eq: k s}
                k_s=\min\{k\in\N\,;\,k\geq |\log s|/\log 2\}.
            \end{equation}
            \item\label{small neg s} If $-1<s<0$, then there holds
            \begin{equation}
            \label{eq: estimate small neg s}
                \||\nabla u|^s\|_{C^{0,\beta}}\leq \lambda_1^s+\frac{2|s|}{\lambda_1^{2|s|}\prod_{j=1}^{k_{|s|}}\lambda_1^{2^{j-1}s}}\|\nabla u\|_{C^{0,\beta}}^{2^{k_{|s|}}|s|}
            \end{equation}
            where $k_{|s|}>0$ is again the constant from \eqref{eq: k s}.
            \item\label{large neg s} If $s\leq -1$, then there holds
            \begin{equation}
            \label{eq: estimate large neg s}
                 \||\nabla u|^s\|_{C^{0,\beta}}\leq \lambda_1^s+2|s|\lambda_1^{2s}\|\nabla u\|_{C^{0,\beta}}^{|s|}.
            \end{equation}
        \end{enumerate}
    \end{lemma}

    \begin{proof}
        Let us note that it is enough to prove \ref{large s}--\ref{small s} as the others are consequences of these. In fact, if we set $t\vcentcolon = -s$ for $s<0$, then we may estimate
        \[
        \begin{split}
             ||\nabla u(x)|^{-t}-|\nabla u(y)|^{-t}|&=\frac{||\nabla u(x)|^{t}-|\nabla u(y)|^{t}|}{|\nabla u(x)|^t |\nabla u(y)|^t}\\
             &\leq \lambda_1^{-2t}||\nabla u(x)|^{t}-|\nabla u(y)|^{t}|\\
             &=\lambda_1^{2s}||\nabla u(x)|^{t}-|\nabla u(y)|^{t}|
        \end{split}
        \]
        and hence
        \[
            [|\nabla u|^{-t}]_{C^{0,\beta}}=\lambda_1^{2s}[|\nabla u|^t]_{C^{0,\beta}}.
        \]
        Using this, it is easy to see that the estimate \eqref{eq: estimate small neg s} can be deduced from \eqref{eq: estimate small s} and the bound  \eqref{eq: estimate large neg s} from \eqref{eq: estimate large s}.

        The estimate \eqref{eq: estimate large s} follows from the triangle inequality for $s=1$ and \eqref{eq: difference rth power} for $s>1$, respectively.

        Finally, to get the bound \eqref{eq: estimate small s} we calculate
        \begin{equation}
        \label{eq: iterated estimate}
        \begin{split}
            ||\nabla u(x)|^s-|\nabla u(y)|^s|&\leq \frac{||\nabla u(x)|^{2s}-|\nabla u(y)|^{2s}|}{|\nabla u(x)|^s+|\nabla u(y)|^s}\\
            &\leq \frac{||\nabla u(x)|^{2s}-|\nabla u(y)|^{2s}|}{2\lambda_1^{s}}\\
            &\leq \frac{||\nabla u(x)|^{4s}-|\nabla u(y)|^{4s}|}{2^2\lambda_1^{s}\lambda_1^{2s}}\\
            &\leq \ldots\\
            &\leq \frac{||\nabla u(x)|^{2^k s}-|\nabla u(y)|^{2^k s}|}{2^k\prod_{j=1}^{k}\lambda_1^{2^{j-1}s}}
        \end{split}
        \end{equation}
        for any $k\in\N$. If we apply \eqref{eq: iterated estimate} with $k=k_s$ (see~\eqref{eq: k s}), which guarantees that $2^ks \geq 1$, and \eqref{eq: difference rth power}, we deduce 
        \[
        \begin{split}
             [|\nabla u|^s]_{C^{0,\beta}}&\leq \frac{2s}{\prod_{j=1}^{k_s}\lambda_1^{2^{j-1}s}}\|\nabla u\|_{L^{\infty}}^{2^{k_s}-1}[\nabla u]_{C^{0,\beta}}\leq \frac{2s}{\prod_{j=1}^{k_s}\lambda_1^{2^{j-1}s}}\|\nabla u\|_{C^{0,\beta}}^{2^{k_s}}.
        \end{split}
        \]
        So, combining this with a trivial bound for the $L^{\infty}$ norm of $|\nabla u|^s$, we get the estimate \eqref{eq: estimate small neg s} and this finishes the proof of Lemma \ref{auxiliary lemma 2}.
    \end{proof}
    
    Next, let us introduce for all $1< r<\infty$ the functions $J^r=(J^r_1,\ldots,J^r_n)\colon \R^n\to \R^n$ by
\begin{equation}
\label{eq: nonlinearity}
    J^r(\xi)=|\xi|^{r-2}\xi.
\end{equation}
A straightforward calculation shows that the Jacobian matrix of $J^r$ is given by
\begin{equation}
\label{eq: derivatives of Jj}
    \nabla_{\xi} J^r(\xi)=|\xi|^{r-2}\left(\mathbf{1}+(r-2)\frac{\xi\otimes \xi}{|\xi|^2}\right)
\end{equation}
for all $\xi\neq 0$. Here, $\mathbf{1}$ denotes the $n\times n$ unit matrix and for all $\eta,\zeta\in\R^n$ the matrix $\eta\otimes \zeta$ has components $(\eta\otimes \zeta)_{ij}=\eta_i\zeta_j$. Using this notation, we can formulate the asymptotic behaviour of $u_\eps$, solving the problem \eqref{eq: epsilon PDE}, as follows:

\begin{proposition}
\label{prop: asymptotic expansion}
    Let $\Omega\subset\R^n$ be a smoothly bounded domain, $1<p<q<\infty$, $0<\alpha\leq 1$, $0<\gamma<\beta<1$ with $\beta \leq \alpha$ and assume that $a\in C^{1,\alpha}(\overline{\Omega})$ is nonnegative. Moreover, suppose that $v\in C^{\infty}(\overline{\Omega})$ is a $p$-harmonic function without critical points. Then the unique solution $u_\eps$, $\eps>0$, to \eqref{eq: epsilon PDE} has the asymptotic expansion
    \begin{equation}
    \label{eq: asymptotic expansion}
         u_\eps=\eps v+\eps^{1+q-p}R_v+o(\eps^{1+q-p})
    \end{equation}
    as $\eps\to 0$ in the sense of $C^{2,\gamma}(\overline{\Omega})$. Here, $R_v\in C^{2,\alpha}(\overline{\Omega})$ denotes the unique solution to 
    \begin{equation}
\label{eq: PDE for R0}
    \begin{cases}
            \Div (A_v^p\nabla R)=-\Div(a|\nabla v|^{q-2}\nabla v)&\text{ in }\Omega,\\
            R=0&\text{ on }\partial\Omega,
        \end{cases}
    \end{equation}
    where $A_v^p=\nabla_{\xi} J^p(\nabla v)$ is uniformly elliptic
    (see~\eqref{eq: derivatives of Jj} and \eqref{eq: coeff p and p,q phase}).
\end{proposition}

\begin{proof}
For any $\eps>0$, let $u_\eps\in W^{1,p,q}_a(\Omega)$ be the unique solution to \eqref{eq: epsilon PDE} (see~Theorem \ref{thm: well-posedness double phase}). Moreover, let us define the functions  $w_\eps\in W^{1,p,q}_{a,0}(\Omega)$ by the Ansatz
\begin{equation}
\label{eq: expansion}
    u_\eps=\eps (v+w_\eps).
\end{equation}
From \eqref{eq: estimate mikko} we deduce that $J^r$ is locally Lipschitz continuous and hence applying the fundamental theorem of calculus to the Lipschitz function $t\mapsto J^r(\xi+t(\zeta-\xi))$ yields the identity
\begin{equation}
\label{Taylor}
    J^r(\zeta)=J^r(\xi)+\int_0^1\nabla_\xi J^r(\xi+t(\zeta-\xi))\,d t\,(\zeta-\xi)
\end{equation}
for all $\zeta,\xi\in\R^n$. Using the expansion \eqref{eq: expansion}, formula \eqref{Taylor} with $\zeta=\nabla u_\eps$, $\xi=\eps\nabla v$ and the $r-2$ homogenity of $\nabla_\xi J^r$ (see~\eqref{eq: derivatives of Jj}), we may expand the $p$ phase and $(p,q)$ phase contributions as 
\begin{equation}
\label{eq: p phase expansion}
\begin{split}
    \Div(|\nabla u_\eps|^{p-2}\nabla u_\eps)&=\eps^{p-1}\Div(|\nabla v|^{p-2}\nabla v)
+\eps^{p-1}\Div (A^p(w_\eps)\nabla w_\eps)\\
    &=\eps^{p-1}\Div (A^p(w_\eps)\nabla w_\eps),
\end{split}
\end{equation}
and 
\begin{equation}
\label{eq: p,q phase expansion}
    \Div(a|\nabla u_\eps|^{q-2}\nabla u_\eps)=\eps^{q-1}\Div(a|\nabla v|^{q-2}\nabla v)
    +\eps^{q-1}\Div (a A^q(w_\eps)\nabla w_\eps),
\end{equation}
where we set
\begin{equation}
\label{eq: coeff p and p,q phase}
    A^r(w_\eps)\vcentcolon =\int_0^1\nabla_{\xi} J^r\left(\nabla v+t\nabla w_\eps\right)\,dt
\end{equation}
for $1<r<\infty$. In \eqref{eq: p phase expansion}, we have also used the fact that $v$ is $p$-harmonic. By using the notation
\begin{equation}
\label{eq: def A eps}
    A_\eps(w_\eps)\vcentcolon =A^p(w_\eps)+\eps^{q-p}a A^q(w_\eps),
\end{equation}
we deduce from \eqref{eq: p phase expansion}, \eqref{eq: p,q phase expansion} and \eqref{eq: epsilon PDE} that $w_\eps\in  W^{1,p,q}_{a,0}(\Omega)$ solves
\begin{equation}
\label{eq: PDE for w eps}
    \begin{cases}
            \Div (A_\eps(w_\eps)\nabla w_\eps)=-\eps^{q-p}\Div(a|\nabla v|^{q-2}\nabla v)&\text{ in }\Omega,\\
            w_\eps=0&\text{ on }\partial\Omega.
        \end{cases}
\end{equation}
The reverse is also true, in that any solution $W_\eps\in  W^{1,p,q}_{a,0}(\Omega)$ of \eqref{eq: PDE for w eps} induces through the formulas \eqref{eq: expansion} and \eqref{Taylor} a solution $U_\eps\in  W^{1,p,q}_{a}(\Omega)$ to \eqref{eq: epsilon PDE}. Hence, there is a one-to-one correspondence between solutions to \eqref{eq: PDE for w eps} and solutions to the double phase problem \eqref{eq: epsilon PDE}. In particular, Theorem \ref{thm: well-posedness double phase} ensures that solutions $W_\eps\in W^{1,p,q}_{a,0}(\Omega)$ to \eqref{eq: PDE for w eps} are unique.

Next, let us fix exponents $0<\gamma<\beta <1$ with $\beta\leq \alpha$ as in the statement, introduce the quantity
\begin{equation}
\label{eq: def of lambda0}
    \lambda_0\vcentcolon =\inf_{x\in \overline{\Omega}}|\nabla v(x)|>0
\end{equation}
and define the convex set
\begin{equation}
\label{eq: small V}
    C^{2,\beta}_{\lambda_0}(\overline{\Omega})\vcentcolon =\{V\in C^{2,\beta}(\overline{\Omega})\,;\,\|V\|_{C^{2,\beta}(\overline{\Omega})}\leq \lambda_0/2\}\subset C^{2,\gamma}(\overline{\Omega}).
\end{equation}
Moreover, let us formally define the map
\begin{equation}
\label{eq: map T}
    \mathcal{T}_\eps\colon C^{2,\beta}_{\lambda_0}(\overline{\Omega})\to C^{2,\beta}_{\lambda_0}(\overline{\Omega})\text{ with }\mathcal{T}_\eps(V)\vcentcolon = W_\eps,
\end{equation}
where $W_\eps \in C^{2,\beta}(\overline{\Omega})$ denotes the solution to
\begin{equation}
\label{eq: quasilinear PDE}
    \begin{cases}
            \Div (A_\eps(V)\nabla W)=-\eps^{q-p}\Div(a|\nabla v|^{q-2}\nabla v)&\text{ in }\Omega,\\
            W=0&\text{ on }\partial\Omega.
        \end{cases}
\end{equation}
Next, we show that $\mathcal{T}_\eps$ is indeed well-defined for small $\eps>0$. To this end, we derive with the help of Lemma \ref{auxiliary lemma 2} various properties of the matrix functions $A_\eps(V)$, which are uniform in  $V\in C^{2,\beta}_{\lambda_0}(\overline{\Omega})$ and $0<\eps\leq 1$. 
\begin{claim}[Properties of $A_\eps(V)$]
\label{claim: properties of A eps V}
    The symmetric matrix functions $A_\eps(V)$ with $V\in C^{2,\beta}_{\lambda_0}(\overline{\Omega})$ and $\eps>0$ have the following properties:
    \begin{enumerate}[(a)]
        \item\label{prop 1 A eps} The matrices $A_\eps(V)$ are uniformly elliptic and satisfy
        \begin{equation}
        \label{eq: uniform ellipticity bound}
            \eta\cdot A_\eps (V)\eta\geq 
            \begin{cases}
                (\lambda_0/2)^{p-2}|\eta|^2&\text{ for }p\geq 2,\\
                (p-1)(\|\nabla v\|_{L^{\infty}}+\lambda_0/2)^{p-2}|\eta|^2&\text{ for } 1<p<2
            \end{cases}
        \end{equation}
        for all nonzero $\eta\in \R^n$, $V\in C^{2,\beta}_{\lambda_0}(\overline{\Omega})$ and $0<\eps\leq 1$.
        \item\label{prop 2 A eps} For any $V\in C^{2,\beta}_{\lambda_0}(\overline{\Omega})$ and $0<\eps\leq 1$, we have $A_\eps(V), \Div(A_\eps(V))\in C^{0,\beta}(\overline{\Omega})$ and there exists a constant
        \begin{equation}
        \label{eq: dependence of constant}
            C_0=C_0(\lambda_0,p,q,\|v\|_{C^{2,\beta}(\overline{\Omega})},(\text{diam}\, \Omega)^{\alpha-\beta},\|a\|_{C^{1,\alpha}(\Omega)})>0
        \end{equation}
        such that
        \begin{equation}
        \label{eq: uniform bound on holder norm of A eps and divergence}
            \|A_\eps(V)\|_{C^{0,\beta}}+\|\Div(A_\eps(V))\|_{C^{0,\beta}}\leq C_0.
        \end{equation}
    \end{enumerate}
\end{claim}

\begin{remark}
\label{rem: uniform bound for small C 1 beta functions}
    The estimate \eqref{eq: uniform ellipticity bound} and the uniform upper bound for $A_\eps(V)$ in \eqref{eq: uniform bound on holder norm of A eps and divergence} also hold, when $V\in C^{1,\beta}(\overline{\Omega})$ satisfies $\|V\|_{C^{1,\beta}(\overline{\Omega})}\leq \lambda_0/2$.
\end{remark}

\begin{proof} The symmetry of the matrix $A_\eps (V)$ is an immediate consequence of \eqref{eq: derivatives of Jj}. \\
    \noindent \ref{prop 1 A eps}: By formula \eqref{eq: derivatives of Jj} we have for all $1<r<\infty$ and $\xi,\eta\neq 0$:
    \begin{equation}
    \label{uniform ellipticity of J}
    \begin{split}
        \eta\cdot \nabla_{\xi} J^r(\xi)\eta&=|\xi|^{r-2}\left(|\eta|^2+(r-2)\frac{|\xi\cdot\eta|^2}{|\xi|^2}\right)\\
        &=|\xi|^{r-2}(1+(r-2)|\hat{\xi}\cdot\hat{\eta}|^2)|\eta|^2\\
        &\geq \min(1,r-1)|\xi|^{r-2}|\eta|^2.
    \end{split}
    \end{equation}
    In the above calculation, we used the notation $\hat{\zeta}=\zeta/|\zeta|$.
    Next, we observe that for all $(x,t)\in\overline{\Omega}\times [0,1]$ and $V\in C^{2,\beta}_{\lambda_0}(\overline{\Omega})$, we have the uniform bounds
    \begin{equation}
    \label{eq: uniform bounds for v0 plus t v}
        \begin{split}
            |\nabla v(x)+t\nabla V(x)|&\geq \lambda_0-t|\nabla V(x)|\geq \lambda_0/2,\\
            |\nabla v(x)+t\nabla V(x)|&\leq \|v\|_{L^{\infty}(\Omega)}+\lambda_0/2.
        \end{split}
    \end{equation}
    Hence, by combining \eqref{uniform ellipticity of J} and \eqref{eq: uniform bounds for v0 plus t v}, we deduce that for all $V\in C^{2,\beta}_{\lambda_0}(\overline{\Omega})$, $(x,t)\in \overline{\Omega}\times [0,1]$ and $\eta\neq 0$ there holds
    \begin{equation}
    \label{eq: uniform ellipticity of J with v0 and v}
        \eta\cdot \nabla_{\xi} J^r(\nabla v(x)+t\nabla V(x))\eta\geq  \begin{cases}
                (\lambda_0/2)^{r-2}|\eta|^2&\text{ for }r\geq 2,\\
                (r-1)(\|\nabla v\|_{L^{\infty}}+\lambda_0/2)^{r-2}|\eta|^2&\text{ for } 1<r<2.
            \end{cases}
    \end{equation}
    Inserting this into \eqref{eq: coeff p and p,q phase}--\eqref{eq: def A eps} and using $a\geq 0$, we obtain the uniform ellipticity bound \eqref{eq: uniform ellipticity bound}.\\

    \noindent\ref{prop 2 A eps}: We first prove $A_\eps (V)\in C^{0,\beta}(\overline{\Omega})$ for all $V\in C^{2,\beta}_{\lambda_0}(\overline{\Omega})$. Let us start by recalling that for all $0\leq \alpha_1\leq \alpha_2\leq 1$ and $u_j\in C^{0,\alpha_j}(\overline{\Omega})$, $j=1,2$, there holds
    \begin{equation}
    \label{eq: product rule Holder}
        \|u_1u_2\|_{C^{0,\alpha_1}}\leq \max(1,(\text{diam }\Omega)^{\alpha_2-\alpha_1})\|u_1\|_{C^{0,\alpha_1}}\|u_2\|_{C^{0,\alpha_2}}.
    \end{equation}
    Now, using \eqref{eq: product rule Holder} and \eqref{eq: coeff p and p,q phase}--\eqref{eq: def A eps}, a direct calculation shows that
    \begin{equation}
    \label{eq: C beta norm of A eps}
        \begin{split}
            &\|A_\eps(V)\|_{C^{0,\beta}}\leq \|A^p(V)\|_{C^{0,\beta}}+\|a\|_{C^{0,\alpha}}\|A^q(V)\|_{C^{0,\beta}}\\
            &\leq \sup_{0\leq t\leq 1}\|\nabla_{\xi} J^p(\nabla v+t\nabla V)\|_{C^{0,\beta}}+\|a\|_{C^{0,\beta}}\sup_{0\leq t\leq 1}\|\nabla_{\xi} J^q(\nabla v+t\nabla V)\|_{C^{0,\beta}}\\
            &\leq \sup_{0\leq t\leq 1}\|\nabla_{\xi} J^p(\nabla v+t\nabla V)\|_{C^{0,\beta}}+(\text{diam }\Omega)^{\alpha-\beta}\|a\|_{C^{0,\alpha}}\sup_{0\leq t\leq 1}\|\nabla_{\xi} J^q(\nabla v+t\nabla V)\|_{C^{0,\beta}}
        \end{split}
    \end{equation}
    for all $0<\eps\leq 1$ and $V\in C^{2,\beta}_{\lambda_0}(\overline{\Omega})$. Hence, it remains to bound $\sup_{0\leq t\leq 1}\|\nabla_{\xi} J^r(\nabla v+t\nabla V)\|_{C^{0,\beta}}$ for $r=p,q$ uniformly to obtain the desired conclusion. By \eqref{eq: product rule Holder}, Lemma \ref{auxiliary lemma 2} with $u\vcentcolon =v_0+tV$ and $\lambda_1=\lambda_0/2$ and \eqref{eq: uniform bounds for v0 plus t v}, we have
    \begin{equation}
    \label{eq: estimate for nabla Jr in Holder}
    \begin{split}
        &\|\nabla_{\xi} J^r(\nabla u)\|_{C^{0,\beta}}\leq \||\nabla u|^{r-2}\|_{C^{0,\beta}}\left\|\mathbf{1}+\frac{\nabla u\otimes \nabla u}{|\nabla u|^2}\right\|_{C^{0,\beta}}\\
        &\leq \||\nabla u|^{r-2}\|_{C^{0,\beta}}\left(1+\|\nabla u\otimes \nabla u\|_{C^{0,\beta}}\||\nabla u|^{-2}\|_{C^{0,\beta}}\right)\\
        &\leq C\||\nabla u|^{r-2}\|_{C^{0,\beta}}\left(1+\|\nabla u\|^2_{C^{0,\beta}}(\lambda_1^{-2}+4\lambda_1^{-4}\|\nabla u\|^2_{C^{0,\beta}})\right)\\
        &\leq \begin{cases}
            C_1(\|\nabla u\|_{C^{0,\beta}}^{r-2}+\|\nabla u\|_{C^{0,\beta}}^r+\|\nabla u\|_{C^{0,\beta}}^{r+2})&\text{ for }r\geq 3,\\
            C_2(\|\nabla u\|_{C^{0,\beta}}^{r-2}+\|\nabla u\|_{C^{0,\beta}}^{2^{\kappa_{r}}(r-2)})(1+\|\nabla u\|_{C^{0,\beta}}^2+\|\nabla u\|_{C^{0,\beta}}^4)&\text{ for }2\leq r< 3,\\
            C_3(1+\|\nabla u\|_{C^{0,\beta}}^{2^{\kappa_{r}}(2-r)})(1+\|\nabla u\|_{C^{0,\beta}}^2+\|\nabla u\|_{C^{0,\beta}}^4)&\text{ for }1< r< 2,
        \end{cases}
    \end{split}
    \end{equation}
    where the constants $C_1,C_2,C_3$ only depend on $\lambda_1$, $r$ and $\kappa_{r}=k_{|r-2|}$, where $k_{|2-r|}$ is the constant from Lemma \ref{auxiliary lemma 2}. 

    Finally, taking into account the uniform bound
    \begin{equation}
        \label{eq: uniform Holder bounds for v0 plus t v}
            \lambda_0/2\leq\|\nabla v+t\nabla V\|_{C^{0,\beta}}\leq \|v\|_{C^{1,\beta}}+\lambda_0/2
    \end{equation}
    for all $V\in C^{2,\beta}_{\lambda_0}(\overline{\Omega})$ and $0\leq t\leq 1$, we establish from the estimates \eqref{eq: C beta norm of A eps} and \eqref{eq: estimate for nabla Jr in Holder} the existence of a constant 
    \[
        C^{(1)}_0=C^{(1)}_0(\lambda_0,p,q,\|\nabla v\|_{C^{0,\beta}},\|a\|_{C^{0,\alpha}},(\text{diam }\Omega)^{\alpha-\beta})>0
    \]
    such that
    \begin{equation}
    \label{eq: uniform bound of A eps V}
        \|A_\eps(V)\|_{C^{0,\beta}}\leq C^{(1)}_0
    \end{equation}
    for all $0<\eps\leq 1$ and $V\in C^{2,\beta}_{\lambda_0}(\overline{\Omega})$.

    Next, we prove the $\beta$ H\"older continuity of $\Div(A_\eps(V))$. First, note that
    \begin{equation}
    \label{eq: decomposition of divergence}
        \Div(A_\eps(V))=\Div(A^p(V))+\eps^{q-p}a\Div(A^q(V))+\eps^{q-p}A^q(V)\nabla a.
    \end{equation}
    From \eqref{eq: uniform bound of A eps V}, \eqref{eq: product rule Holder} and $a\in C^{1,\alpha}(\overline{\Omega})$, we already know that there exists a constant
    \begin{equation}     
    \label{eq: first constant for divergence}
    C^{(2)}_0=C^{(2)}_0(\lambda_0,p,q,\|\nabla v\|_{C^{0,\beta}},\|a\|_{C^{1,\alpha}},(\text{diam }\Omega)^{\alpha-\beta})>0
    \end{equation}
    such that
    \begin{equation}
    \label{eq: uniform bound of last term of divergence}
        \eps^{q-p}\|A^q(V)\nabla a\|_{C^{0,\beta}}\leq (\text{diam }\Omega)^{\alpha-\beta}\|A^q(V)\|_{C^{0,\beta}}\|a\|_{C^{1,\alpha}}\leq C_0^{(2)}
    \end{equation}
    for all $0<\eps\leq 1$ and $V\in C^{2,\beta}_{\lambda_0}(\overline{\Omega})$. Hence, it remains to uniformly bound $\|\Div(A^r(V))\|_{C^{0,\beta}}$, $1<r<\infty$, by some constant $C_0^{(3)}$, having the same dependence structure as $C_0$ in \eqref{eq: dependence of constant},
    to deduce the existence of a constant $C_2>0$ with \eqref{eq: dependence of constant} such that the estimate \eqref{eq: uniform bound on holder norm of A eps and divergence} holds.

    To achieve this, for a given function $v\in C^{2,\beta}(\overline{\Omega})$ without critical points and $1\leq i,j,k\leq n$, we calculate
    \begin{equation}
    \label{eq: second derivative of J}
    \begin{split}
        \partial_{i}\partial_{\xi_j}J^r_k(\nabla v)&=\partial_i\left(|\nabla v|^{r-2}\delta_{jk}+(r-2)|\nabla v|^{r-4}\partial_j v\partial_k v\right)\\
        &=(r-2)|\nabla v|^{r-4}( \delta_{jk}\partial_{\ell} v \partial_{i\ell}v+\partial_{ij} v\partial_k v+\partial_j v\partial_{ik} v)\\
        &\quad +(r-4)(r-2)|\nabla v|^{r-6}\partial_j v\partial_k v\partial_{\ell} v \partial_{i\ell}v.
    \end{split}
    \end{equation}
    Here and in the rest of this article, we are adopting the Einstein summation convention. In the first equality we used \eqref{eq: derivatives of Jj} and in the second equality the identity 
    \begin{equation}
    \label{eq: derivative of mod power s}
        \partial_i |\nabla v|^s=s |\nabla v|^{s-2}\partial_{\ell} v \partial_{i\ell}v
    \end{equation}
    for all $s\in\R$. Now, suppose again that
    \[
        \inf_{x\in \overline{\Omega}}|\nabla v(x)|\geq \lambda_1
    \]
    for some $\lambda_1>0$. Using \eqref{eq: second derivative of J}, we deduce
    \begin{equation}
    \label{eq: divergence of nabla J}
    \begin{split}
        (\Div \nabla_\xi J^r(\nabla v))_i&=\partial_j(\nabla_\xi J(\nabla v))_{ij}=\partial_j \partial_{\xi_j}J^r_i(\nabla v)\\
        &=\partial_j\left(|\nabla v|^{r-2}\delta_{ij}+(r-2)|\nabla v|^{r-4}\partial_i v\partial_j v\right)\\
        &=(r-2)|\nabla v|^{r-4}( \delta_{ij}\partial_{\ell} v \partial_{j\ell}v+\partial_{jj} v\partial_i v+\partial_j v\partial_{ji} v)\\
        &\quad +(r-2)(r-4)|\nabla v|^{r-6}\partial_j v\partial_i v\partial_{\ell} v \partial_{j\ell}v
    \end{split}
    \end{equation}
    for every $1\leq i\leq n$. Now, applying the estimate \eqref{eq: product rule Holder} and Lemma \ref{auxiliary lemma 2}, we obtain
    \begin{equation}
        \begin{split}
            &\|\Div \nabla_\xi J^r(\nabla v)\|_{C^{0,\beta}}\\
            &\leq C(\||\nabla v|^{r-4}\|_{C^{0,\beta}}+\||\nabla v|^{r-6}\|_{C^{0,\beta}}\|\nabla v\|_{C^{0,\beta}}^2)\|\nabla v\|_{C^{0,\beta}}\|\nabla^2 v\|_{C^{0,\beta}}\\
            &\leq C\||\nabla v|^{r-4}\|_{C^{0,\beta}}(1+\||\nabla v|^{-2}\|_{C^{0,\beta}}\|\nabla v\|_{C^{0,\beta}}^2)\|\nabla v\|_{C^{0,\beta}}\|\nabla^2 v\|_{C^{0,\beta}}\\
            &\leq C\||\nabla v|^{r-4}\|_{C^{0,\beta}}(1+\|\nabla v\|_{C^{0,\beta}}^4)\|\nabla v\|_{C^{0,\beta}}\|\nabla^2 v\|_{C^{0,\beta}}\\
            &\leq C\||\nabla v|^{r-4}\|_{C^{0,\beta}}(1+\|\nabla v\|_{C^{0,\beta}}^4)\|\nabla^2 v\|_{C^{0,\beta}}\\
            &\leq \begin{cases}
            c_1\|\nabla v\|_{C^{0,\beta}}^{r-4}(1+\|\nabla v\|_{C^{0,\beta}}^4)\|\nabla^2 v\|_{C^{0,\beta}}&\text{ for }r\geq 5,\\
            c_2(\|\nabla v\|_{C^{0,\beta}}^{r-2}+\|\nabla v\|_{C^{0,\beta}}^{2^{\kappa_{r}}(r-2)})(1+\|\nabla v\|_{C^{0,\beta}}^4)\|\nabla^2 v\|_{C^{0,\beta}}&\text{ for }4\leq r< 5,\\
            c_3(1+\|\nabla v\|_{C^{0,\beta}}^{2^{\kappa_{r}}(2-r)})(1+\|\nabla v\|_{C^{0,\beta}}^4)\|\nabla^2 v\|_{C^{0,\beta}}&\text{ for }1< r< 4,
        \end{cases}
        \end{split}
    \end{equation}
    where  $c_1,c_2,c_3>0$ only depend on $\lambda_1$, $r$ and $\kappa_r$ is the same as in \eqref{eq: estimate for nabla Jr in Holder}. Now, we can argue as in the case of $A_\eps (V)$ (see~\eqref{eq: C beta norm of A eps}), while using 
    \[
        \lambda_0/2\leq\|\nabla v+t\nabla V\|_{C^{1,\beta}}\leq \|v\|_{C^{2,\beta}}+\lambda_0/2,
    \]
    to deduce the desired uniform bound
    \begin{equation}
    \label{eq: holder estimate of remaining divergence terms}
    \begin{split}
        &\|\Div(A^p(V))\|_{C^{0,\beta}}+\eps^{q-p}\|a\Div(A^p(V))\|_{C^{0,\beta}}\\
        &\leq \|\Div(A^p(V))\|_{C^{0,\beta}}+\|a\|_{C^{0,\beta}}\|\Div(A^p(V))\|_{C^{0,\beta}}\\
        &\leq \|\Div(A^p(V))\|_{C^{0,\beta}}+(\text{diam }\Omega)^{\alpha-\beta}\|a\|_{C^{0,\alpha}}\|\Div(A^p(V))\|_{C^{0,\beta}}\\
        &\leq C_0^{(3)}
    \end{split}
    \end{equation}
    for all $V\in C^{2,\beta}_{\lambda_0}(\overline{\Omega})$, $0<\eps\leq 1$ and some constant
    \[
        C_0^{(3)}=C_0^{(3)}(\lambda_0,p,q,\|v\|_{C^{2,\beta}},(\text{diam}\, \Omega)^{\alpha-\beta},\|a\|_{C^{1,\alpha}(\Omega)})>0.
    \]
    Hence, by the estimates \eqref{eq: uniform bound of A eps V}, \eqref{eq: uniform bound of last term of divergence} and \eqref{eq: holder estimate of remaining divergence terms}, we see that we have proved \eqref{eq: uniform bound on holder norm of A eps and divergence} with $C_0=\sum_{j=1}^{3}C_0^{(j)}$. This finishes the proof of Claim \ref{claim: properties of A eps V}.
\end{proof}
Now, by expanding the partial differential operator $\Div (A_\eps(V)\nabla W)$ as 
\begin{equation}
\label{eq: expanded PDO}
    \Div (A_\eps(V)\nabla W)=A_\eps(V):\nabla^2 W+\Div(A_\eps(V))\cdot \nabla W,
\end{equation}
where we set $A:B=A_{ij}B_{ij}$ for two matrices $A,B\in\R^{n\times n}$, and using Claim \ref{claim: properties of A eps V}, we can apply \cite[Theorem 6.14]{GiTru} for any $V\in C^{2,\beta}_{\lambda_0}(\overline{\Omega})$ and $0<\eps\leq 1$ to deduce the existence of a unique solution $W_\eps\in C^{2,\beta}(\overline{\Omega})$ to \eqref{eq: quasilinear PDE}. By \cite[Theorem 6.6]{GiTru} this unique solution $W_\eps$ satisfies
\begin{equation}
    \|W_\eps\|_{C^{2,\beta}(\overline{\Omega})}\leq C(\|W_\eps\|_{L^{\infty}(\Omega)}+\eps^{q-p}\|f_0\|_{C^{0,\beta}(\overline{\Omega})})
\end{equation}
for some $C>0$ only depending on $n,\Omega,\alpha,\beta,p,q,\lambda_0,\|v\|_{C^{2,\beta}(\overline{\Omega})}$ and $\|a\|_{C^{1,\alpha}(\overline{\Omega})}$, where we set
\begin{equation}
    f_0\vcentcolon = -\Div(a|\nabla v|^{q-2}\nabla v).
\end{equation}
Furthermore, from \cite[Theorem 3.7]{GiTru} and Claim \ref{claim: properties of A eps V} we get the estimate
\begin{equation}
    \|W_\eps\|_{L^{\infty}(\Omega)}\leq c\eps^{q-p}\|f_0\|_{L^{\infty}(\Omega)}
\end{equation}
for some constant $c>0$ only depending on $\lambda_0,p,q,\alpha,\beta,\|v\|_{C^{2,\beta}(\overline{\Omega})},\text{diam}\, \Omega$ and $\|a\|_{C^{1,\alpha}(\Omega)}$. Therefore, the solution $W_\eps$ satisfies 
\begin{equation}
\label{eq: uniform bound in eps}
    \|W_\eps\|_{C^{2,\beta}(\overline{\Omega})}\leq C_1\eps^{q-p}\|f_0\|_{C^{0,\beta}(\overline{\Omega})},
\end{equation}
where $C_1>0$ only depends on $\lambda_0,p,q,\alpha,\beta,\|v\|_{C^{2,\beta}(\overline{\Omega})},\text{diam}\, \Omega$ and $\|a\|_{C^{1,\alpha}(\Omega)}$.

Thus, if we choose $\eps_0>0$ such that
\begin{equation}
\label{eq: choice of small eps}
    C_1\eps_0^{q-p}\|f_0\|_{C^{0,\beta}(\overline{\Omega})}\leq \lambda_0/2,
\end{equation}
then the map $\mathcal{T}_\eps$, with $0<\eps\leq \eps_0$ and defined by formula \eqref{eq: map T}, maps $C^{2,\beta}_{\lambda_0}(\overline{\Omega})$ to itself and so $\mathcal{T}_\eps$ is well-defined for all $0<\eps\leq \eps_0$. 

Next, we show that $C^{2,\beta}_{\lambda_0}(\overline{\Omega})$ is compact in the topology inherited from $C^{2,\gamma}(\overline{\Omega})$. To this end, let us recall that the embedding
\begin{equation}
\label{eq: compact embedding}
    C^{2,\beta}(\overline{\Omega})\hookrightarrow C^{2,\gamma}(\overline{\Omega})
\end{equation}
is compact and so $C^{2,\beta}_{\lambda_0}(\overline{\Omega})$ is precompact in $C^{2,\gamma}(\overline{\Omega})$. By using the Arzel\'a--Ascoli theorem, we can conclude that $C^{2,\beta}_{\lambda_0}(\overline{\Omega})$ is closed in $C^{2,\gamma}(\overline{\Omega})$ and so compact in $C^{2,\gamma}(\overline{\Omega})$. To see that it is closed, let $(V_k)_{k\in\N}\subset C^{2,\beta}_{\lambda_0}(\overline{\Omega})$ and assume that $V_k\to V$ in $C^{2,\gamma}(\overline{\Omega})$ as $k\to\infty$. Then \cite[Section 1.1, (H8)]{Elliptic-PDEs-Xavi}, which is a consequence of the Arzel\'a--Ascoli theorem, implies that $V\in C^{2,\beta}(\overline{\Omega})$ and $\|V\|_{C^{2,\beta}(\overline{\Omega})}\leq \lambda_0/2$. Therefore, we have $V\in C^{2,\beta}_{\lambda_0}(\overline{\Omega})$ and so the assertion follows.

Finally, we show that the map $\mathcal{T}_\eps$ is continuous for all $0<\eps\leq \eps_0$, where $C^{2,\beta}_{\lambda_0}(\overline{\Omega})$ is considered again as a subset of $C^{2,\gamma}(\overline{\Omega})$. Hence, let $(V_k)_{k\in\N}\subset C^{2,\beta}_{\lambda_0}(\overline{\Omega})$ and assume that $V_k\to V$ in $C^{2,\gamma}(\overline{\Omega})$ as $k\to\infty$ for some $V\in C^{2,\gamma}(\overline{\Omega})$. As $C^{2,\beta}_{\lambda_0}(\overline{\Omega})$ is closed, we have $V\in C^{2,\beta}_{\lambda_0}(\overline{\Omega})$. Moreover, let us denote the corresponding solutions to \eqref{eq: quasilinear PDE} by $W_k$ and $W$, respectively. Recall that by the arguments above these solutions belong to $C^{2,\beta}_{\lambda_0}(\overline{\Omega})$, when $0<\eps\leq \eps_0$ (see~Claim \ref{claim: properties of A eps V} and \cite[Theorem 6.14]{GiTru}). Arguing as for \eqref{eq: uniform bound in eps}, we deduce that $\|W_k\|_{C^{2,\beta}(\overline{\Omega})}$ is uniformly bounded in $k$ and hence compactness of the embedding \eqref{eq: compact embedding} implies that $W_k\to \widetilde{W}$ in $C^{2,\gamma}(\overline{\Omega})$ as $k\to\infty$ (up to subsequences) for some $\widetilde{W}\in C^{2,\beta}(\overline{\Omega})$ satisfying the same bounds as $W_k$ (see~\cite[Section 1.1, (H8)]{Elliptic-PDEs-Xavi}). Then, we may deduce that
\begin{equation}
\label{eq: convergece for continuity of T eps}
\begin{split}
   \Div(A_\eps(V_k)\nabla W_k)\to \Div(A_\eps (V)\nabla \widetilde{W})
\end{split}
\end{equation}
and hence by uniqueness of solutions to \eqref{eq: quasilinear PDE} there holds $\widetilde{W}=W$. In \eqref{eq: convergece for continuity of T eps}, we are using that the divergence operator can be written in the form \eqref{eq: expanded PDO} and that $V\mapsto A_\eps (V),\Div(A_\eps(V))$ is continuous from $C^2(\overline{\Omega})$ to $C(\overline{\Omega})$ as long as $\|V\|_{C^{2,\beta}(\overline{\Omega})}\leq \lambda_0/2$, which follows from \eqref{eq: coeff p and p,q phase}, \eqref{eq: divergence of nabla J}, Claim \ref{claim: properties of A eps V} and the dominated convergence theorem. As the limit does not depend on the extracted subsequence, we can conclude that the whole sequence $W_k$ converges to $W$ in $C^{2,\gamma}(\overline{\Omega})$. Thus, the map $\mathcal{T}_\eps$ is continuous for all $0<\eps\leq \eps_0$.

Therefore, by the above observations on the maps $\mathcal{T}_\eps$, $0<\eps\leq \eps_0$, we can apply Schauder's fixed point theorem \cite[Theorem 11.1]{GiTru} to deduce that for any $0<\eps\leq \eps_0$ the map $\mathcal{T}_\eps$ has a fixed point $W_\eps$ in $C^{2,\beta}_{\lambda_0}(\overline{\Omega})$. Furthermore, by uniqueness of solutions to \eqref{eq: PDE for w eps} and $C^{2,\beta}(\overline{\Omega})\subset W^{1,p,q}_{a}(\Omega)$ we must have 
\begin{equation}
\label{eq: increased regularity of w eps}
    w_\eps=W_\eps\in C^{2,\beta}_{\lambda_0}(\overline{\Omega}),
\end{equation}
where $w_\eps$ is defined by \eqref{eq: expansion}. Hence, \eqref{eq: uniform bound in eps} and \eqref{eq: increased regularity of w eps} ensure that there holds
\begin{equation}
\label{eq: uniform bound of w eps}
    \| w_\eps\|_{C^{2,\beta}(\overline{\Omega})}\leq C_0\eps^{q-p}
\end{equation}
for all $0<\eps\leq \eps_0$.

Because of the asymptotic behaviour \eqref{eq: uniform bound of w eps}, we now introduce the functions
\begin{equation}
\label{eq: def of Reps}
    R_\eps=\eps^{p-q}w_\eps    
\end{equation}
for all $0<\eps\leq \eps_0$. Moreover, the estimate \eqref{eq: uniform bound of w eps} demonstrates that $R_\eps$, $0<\eps\leq \eps_0$, is uniformly bounded in $C^{2,\beta}(\overline{\Omega})$ and hence, as above, we can conclude that there exists $R_v\in  C^{2,\beta}(\overline{\Omega})$ such that
\begin{equation}
\label{eq: limit R eps}
    R_{\eps_k}\to R_v\text{ in }C^{2,\gamma}(\overline{\Omega})
\end{equation}
as $k\to\infty$ for a suitable subsequence $(\eps_k)_{k\in\N}\subset (0,\eps_0]$. Therefore, using \eqref{eq: uniform bound of w eps}, $q>p$ and \eqref{eq: limit R eps}, we see that passing to the limit in \eqref{eq: PDE for w eps}, and noting that $A_v^p=A^p(0)$,  yields
\begin{equation}
\label{eq: PDE for R0 proof}
    \begin{cases}
            \Div (A_v^p\nabla R_v)=-\Div(a|\nabla v|^{q-2}\nabla v)&\text{ in }\Omega,\\
            R_v=0&\text{ on }\partial\Omega.
        \end{cases}
\end{equation}
To see that we obtain in the limit the PDE \eqref{eq: PDE for R0 proof}, one can argue similarly as for the continuity of the operators $\mathcal{T}_\eps$, $0<\eps\leq \eps_0$.

Since $R_v\in C^{2,\beta}(\overline{\Omega})$ is the unique solution of \eqref{eq: PDE for R0 proof}, it does not depend on the subsequence taken and hence the whole sequence $(R_\eps)_{0<\eps\leq \eps_0}$ needs to converge to $R_v$. Furthermore, for example by \cite[Theorem 6.19]{GiTru}, it follows that $R_v\in C^{2,\alpha}(\overline{\Omega})$. Therefore, we have shown the desired asymptotic expansion \eqref{eq: asymptotic expansion}.
\end{proof}

\subsection{Asymptotic expansion of solutions for \texorpdfstring{$p>q$}{p greater q}}
\label{subsec: Asymptotic expansion of solutions to the double phase equation with large Dirichlet datum}

The goal of this section is to prove that for any smooth $p$-harmonic function $v$ without critical points and $1<q<p<\infty$, the solution $u_\mu$ to 
    \begin{equation}
    \label{eq: mu PDE}
    \begin{cases}
            \Div(|\nabla u|^{p-2}\nabla u+a|\nabla u|^{q-2}\nabla u)=0&\text{ in }\Omega,\\
            u=\mu v&\text{ on }\partial\Omega,
        \end{cases}
    \end{equation}
    can be expanded as
    \begin{equation}
    \label{eq: asymptotics p>q}
        u_\mu=\mu v+\mu^{1+q-p}R_v+o(\mu^{1+q-p})
    \end{equation}
    for a suitable function $R_v$ as $\mu\to\infty$. Thus, if the $p$-Laplacian is the dominant term, then we consider large boundary values instead of small ones as used in the range $p<q$. Note that in the latter case the weighted $q-$Laplacian yields the major contribution.

    Let us start by observing several facts for the double phase problem \eqref{eq: mu PDE}, when $q<p$ and the boundary condition $\mu v$ is replaced by an arbitrary function $f\colon\partial\Omega\to \R$: 
    \begin{enumerate}[({F}1)]
        \item\label{solution space} \textit{Function space:} If $1<q<p<\infty$, then $W^{1,p,q}_{a}(\Omega)=W^{1,p}(\Omega)$.
        \item\label{well-posedness} \textit{Well-posedness:} Theorem \ref{thm: well-posedness double phase} remains valid for $1<q<p<\infty$, when $f\in W^{1,p}(\Omega)\hookrightarrow W^{1,q}(\Omega)$. More concretely, in this case \cite[Theorem 3.30]{dacorogna2007direct} and the arguments in the proof of Theorem \ref{thm: well-posedness double phase} ensure the existence of a unique solution $u\in W^{1,p}(\Omega)$ of the double phase problem with boundary value $f$, which coincides with the unique minimizer of $\mathcal{F}$ over the set $M_f=f+W^{1,p}_0(\Omega)$.
        \item\label{maximum principle} \textit{Comparison and maximum principle:} The comparison principle (Proposition \ref{prop: comparison}) still holds in the case $1<q<p<\infty$ and as explained in Section \ref{sec: maximum principle} this in turn guarantees the validity of the maximum principle. Thus, if $u\in W^{1,p}(\Omega)$ solves the double phase problem with boundary condition $f\in L^{\infty}(\Omega)\cap W^{1,p}(\Omega)$, then one has
        \begin{equation}
        \label{eq: maximum principle q<p}
            \|u\|_{L^{\infty}(\Omega)}\leq \|f\|_{L^{\infty}(\Omega)}.
        \end{equation}
    \end{enumerate}

Next, we show the asymptotic expansion \eqref{eq: asymptotics p>q}. More precisely, we have the following result:

    \begin{proposition}
\label{prop: asymptotic expansion p>q}
    Let $\Omega\subset\R^n$ be a smoothly bounded domain, $1<q<p<\infty$, $0<\alpha\leq 1$, $0<\gamma<\beta<1$ with $\beta \leq \alpha$ and assume that $a\in C^{1,\alpha}(\overline{\Omega})$ is nonnegative. Moreover, suppose that $v\in C^{\infty}(\overline{\Omega})$ is a $p$-harmonic function without critical points. Then the unique solution $u_\mu$, $\mu>0$, to \eqref{eq: mu PDE} has the asymptotic expansion \eqref{eq: asymptotics p>q} in the sense of $C^{2,\gamma}(\overline{\Omega})$ as $\mu\to \infty$, in which $R_v\in C^{2,\alpha}(\overline{\Omega})$ again denotes the unique solution to 
    \begin{equation}
\label{eq: PDE for R0 p>q}
    \begin{cases}
            \Div (A_v^p\nabla R)=-\Div(a|\nabla v|^{q-2}\nabla v)&\text{ in }\Omega,\\
            R=0&\text{ on }\partial\Omega,
        \end{cases}
    \end{equation}
    where $A_v^p$ is the uniformly elliptic matrix $\nabla_{\xi} J^p(\nabla v)$
    (see~\eqref{eq: derivatives of Jj} and \eqref{eq: coeff p and p,q phase}).
\end{proposition}

\begin{proof}
    Since the proof of Proposition \ref{prop: asymptotic expansion p>q} is very similar to the one of Proposition \ref{prop: asymptotic expansion} we omit some of the details.

    Assume $v\in C^\infty(\overline{\Omega})$ is any $p$-harmonic function without any critical points and let us denote by $u_\mu\in W^{1,p}(\Omega)$, $\mu>0$, the unique solution to \eqref{eq: mu PDE} (see~\ref{well-posedness}). By the maximum principle we know that there holds
    \begin{equation}
        \|u_\mu\|_{L^{\infty}(\Omega)}\lesssim \mu \|v\|_{L^{\infty}(\Omega)}
    \end{equation}
    (see~\ref{maximum principle}) and thus we define $w_\mu\in W^{1,p}_{0}(\Omega)$ by the Ansatz
    \begin{equation}
    \label{eq: expansion large dirichlet data}
        u_\mu=\mu (v+w_\mu),
    \end{equation}
    which solves 
    \begin{equation}
    \label{eq: PDE for w mu}
    \begin{cases}
            \Div (A_\mu(w_\mu)\nabla w_\mu)=-\mu^{q-p}\Div(a|\nabla v|^{q-2}\nabla v)&\text{ in }\Omega,\\
            w_\mu=0&\text{ on }\partial\Omega,
        \end{cases}
\end{equation}
    where 
    \[
        A_\mu(w_\mu)=A^p(w_\mu)+\mu^{q-p}a\,A^q(w_\mu)
    \]
    and $A^r(w)$, $1<r<\infty$, is defined as in \eqref{eq: coeff p and p,q phase} (see~\eqref{eq: p phase expansion}, \eqref{eq: p,q phase expansion} and \eqref{eq: PDE for w eps}). So, we end up with precisely the same problem as for $w_\eps$, but with the replacement $\eps\to\mu$ and $w_\eps\to w_\mu$. Furthermore, note that in Claim \ref{claim: properties of A eps V} the ordering of the exponents $p,q$ has no influence and thus the assertions remain valid as long as $\mu\geq 1$. Hence, by following the arguments in the proof of Proposition \ref{prop: asymptotic expansion} (essentially Schauder's fixed point theorem) we may deduce that for $\mu\geq \mu_0$, with $\mu_0\geq 1$ large, the solution $w_\mu$ of \eqref{eq: PDE for w mu} is of class $C^{2,\beta}_{\lambda_0}(\overline{\Omega})$ and there holds 
    \begin{equation}
    \label{eq: uniform estimate w mu}
        \|w_\mu\|_{C^{2,\beta}(\overline{\Omega})}\lesssim \mu^{q-p}
    \end{equation}
    for any $0<\beta<\min(1,\alpha)$ (see proof of Proposition \ref{prop: asymptotic expansion}). Thus, we deduce that the function $R_\mu=\mu^{p-q}w_\mu$ are uniformly bounded in $C^{2,\beta}(\overline{\Omega})$ for $\mu\geq \mu_0$, with $\mu_0\geq 1$ large. By \eqref{eq: uniform estimate w mu} and the usual compactness argument we deduce that there exists $R_v\in C^{2,\beta}(\overline{\Omega})$ such that
    \[
        R_\mu\to R_v\text{ in }C^{0,\gamma}(\overline{\Omega})
    \]
    as $\mu\to\infty$ (up to the extraction of a subsequence), for any $0<\gamma<\beta$. This function $R_v$ again satisfies \eqref{eq: PDE for R0 p>q} (as $p>q$ and $w_\mu\to 0$ in $C^{2,\beta}(\overline{\Omega})$) and hence by elliptic regularity theory it follows that $R_v\in C^{2,\alpha}(\overline{\Omega})$ and the full sequence $R_\mu$, $\mu\geq \mu_0$, converges to $R_v$. This already ensures the desired asymptotic expansion \eqref{eq: asymptotics p>q}.
\end{proof}

\section{DN map of the double phase problem}
\label{sec: DN map double phase}

   In this section, we rigorously introduce the Dirichlet to Neumann map related to double phase problem \eqref{eq: PDE double phase}, which we defined formally in  \eqref{eq: formal DN map}. 

   First, let us note the following simple lemma.

   \begin{lemma}
   \label{eq: same boundary data}
       Let $\Omega\subset\R^n$ be a bounded Lipschitz domain, $1<p\leq q<\infty$ and $a\in L^{\infty}(\Omega)$ a nonnegative function. Assume that $u_j\in W^{1,p,q}_a(\Omega)$ is the unique solution to \eqref{eq: PDE double phase} with $f=f_j\in W^{1,q}(\Omega)$ for $j=1,2$ (see~Theorem \ref{thm: well-posedness double phase}). If $f_1=f_2$ on $\partial\Omega$, then $u_1=u_2$ in $\Omega$. Hence, in particular, for any $f\in W^{1-1/q,q}(\partial\Omega)$ there exists a unique solution $u\in W^{1,p,q}_a(\Omega)$ to \eqref{eq: PDE double phase}. Furthermore, we have the following continuity estimates
       \begin{equation}
        \label{eq: energy estimate W 1p bdry value}
            \|u\|_{W^{1,p}(\Omega)}\leq C(\|f\|_{W^{1-1/q,q}(\partial\Omega)}+\|f\|_{W^{1-1/q,q}(\partial\Omega)}^{q/p})
        \end{equation}
        and
       \begin{equation}
       \label{eq: continuity estimate bdry value}
           \|u\|_{W^{1,p,q}_a(\Omega)}\leq C(\|f\|_{W^{1-1/q,q}(\partial\Omega)}+\|f\|_{W^{1-1/q,q}(\partial\Omega)}^{p/q}+\|f\|_{W^{1-1/q,q}(\partial\Omega)}^{q/p}).
       \end{equation}
   \end{lemma}

   \begin{proof}
       The first part follows from the comparison principle (Proposition \ref{prop: comparison}). The second part of the lemma is an immediate consequence of Theorem \ref{thm: well-posedness double phase}, the previous established independence of the extension of the boundary value $f|_{\partial\Omega}$ and the fact that the image of the trace operator $\text{tr}\colon W^{1,q}(\Omega)\to L^q(\partial\Omega)$ coincides with $W^{1-1/q,q}(\partial\Omega)$. Finally, to see the continuity estimates, suppose that $F\in W^{1,q}(\Omega)$ is an extension of $f\in W^{1-1/q,q}(\partial\Omega)$. By applying \eqref{eq: energy estimate W 1p} and \eqref{eq: energy estimate}, respectively, and using the continuity of the extension operator $E\colon W^{1-1/q,q}(\partial\Omega)\to W^{1,q}(\Omega)$, we get the estimates \eqref{eq: energy estimate W 1p bdry value} and \eqref{eq: continuity estimate bdry value}. Therefore, we can conclude the proof.
   \end{proof}

   \begin{proposition}[DN map]
   \label{prop: Dn map}
       Let $\Omega\subset\R^n$ be a bounded Lipschitz domain, $1<p\leq q<\infty$ and $a\in L^{\infty}(\Omega)$ a nonnegative function. Then we define the \emph{Dirichlet to Neumann (DN) map} $\Lambda_a$ by
       \begin{equation}
       \label{eq: rigorous DN map}
           \langle \Lambda_a f,g\rangle\vcentcolon = \int_{\Omega}\left(|\nabla u|^{p-2}\nabla u+a|\nabla u|^{q-2}\nabla u\right)\cdot\nabla \omega\,dx
       \end{equation}
       for all $f,g\in W^{1-1/q,q}(\partial\Omega)$, where $u\in W^{1,p,q}_a(\Omega)$ is the unique solution of \eqref{eq: PDE double phase} and $\omega\in W^{1,q}(\Omega)$ is an extension of $g\in W^{1-1/q,q}(\partial\Omega)$. 
   \end{proposition}

    \begin{proof}
        First, recall that by Lemma \ref{eq: same boundary data} there exists a unique solution $u\in W^{1,p,q}_a(\Omega)$ to \eqref{eq: PDE double phase}, for any $f\in W^{1-1/q,q}(\partial\Omega)$. Next, let $\omega\in W^{1,q}(\Omega)$ be any extension of $g\in W^{1-1/q,q}(\partial\Omega)$. Note that the integral in \eqref{eq: rigorous DN map} exists by H\"older's inequality with $\frac{p-1}{p}+\frac{1}{p}=1$ and $\frac{q-1}{q}+\frac{1}{q}=1$, respectively. In fact, we may estimate
        \[
        \begin{split}
            |\langle \Lambda_a f,g\rangle|&\leq \|\nabla u\|_{L^p(\Omega)}^{p-1}\|\nabla \omega\|_{L^p(\Omega)}+\|\nabla u\|_{L^q(\Omega;a)}^{q-1}\|\nabla \omega\|_{L^q(\Omega;a)}\\
            &\leq (\|\nabla u\|_{L^p(\Omega)}^{p-1}+\|a\|_{L^{\infty}(\Omega)}^{1/q}\|\nabla u\|_{L^q(\Omega;a)}^{q-1})\|\nabla \omega\|_{L^q(\Omega)}<\infty.
        \end{split}
        \] 
        Moreover, the value of $\langle \Lambda_a f,g\rangle$ does not depend on the extension $\omega\in W^{1,q}(\Omega)$ of $g$ as $u$ solves \eqref{eq: PDE double phase} (see~\eqref{eq: weak solutions double phase}).
    \end{proof}

    \begin{remark}[DN map for $p>q$]
    \label{remark: DN map for p>q}
        The results of Lemma  \ref{eq: same boundary data} and Proposition \ref{prop: Dn map} continue to hold for $p>q$, when one chooses boundary values $f,g\in W^{1-1/p,p}(\partial\Omega)$ (see~\ref{solution space}--\ref{maximum principle}).
    \end{remark}

\section{Linearization of the $p$-Laplace equation}
\label{sec: p-Laplace linearization}

This section is concerned with the construction of families of solutions $(v_\tau)$ to the $p$-Laplace equation
\begin{equation}
\label{eq-p-Laplace}
 \begin{cases}
            \Div(|\nabla v|^{p-2}\nabla v)=0&\text{ in }\Omega,\\
            v=f&\text{ on }\partial\Omega
        \end{cases}
\end{equation}
with a prescribed zeroth order term. More precisely, we shall prove the following lemma.
\begin{lemma}
\label{lemma: family of p harmonic functions}
    Let $\Omega\subset\R^n$ be a smoothly bounded domain and $1<p<\infty$ and $0<\alpha\leq 1$. Moreover, suppose that $v_0\in C^{\infty}(\overline{\Omega})$ is any $p$-harmonic function without critical points, $\phi\in C^{\infty}(\partial\Omega)$ and $V$ is the solution to the Dirichlet problem 
    \begin{equation}
    \label{eq-V}
    \begin{cases}
        \Div(A^p_{v_0}\nabla V)=0& \text{ in }\Omega,\\
        V=\phi&\text{ on }\partial\Omega,
    \end{cases}
    \end{equation}
    then there exists an exponent $0<\beta'<1$ with the following property: For any $0<\gamma'<\beta'<1$, there exists $\tau_0>0$ and a family $(v_\tau)_{\tau\in [-\tau_0,\tau_0]}\subset C^{1,\gamma'}(\overline{\Omega})$ of $p$-harmonic functions that have the asymptotic expansion
    \begin{equation}
    \label{eq: v tau expansion}
        v_\tau =v_0+\tau V+o(\tau),
    \end{equation}
    in the sense of $C^{1,\gamma'}(\overline{\Omega})$ as $\tau\to 0$.
\end{lemma}

\begin{remark}
    Note that \eqref{eq-p-Laplace} has smooth solutions without critical points in $\Omega$. For example, any non-trivial affine function is such a solution. 
\end{remark}

\begin{remark}
    Observe that if $v_0\in C^{\infty}(\overline{\Omega})$ is a fixed $p$-harmonic function without critical points and $V$ is any smooth solution of $\Div(A^p_{v_0}\nabla V)=0$ in $\Omega$, then we can always find a corresponding family $(v_\tau)$ of $p$-harmonic functions having the asymptotic expansion \eqref{eq: v tau expansion}. Furthermore, let us point out that the assertion of Lemma \ref{lemma: family of p harmonic functions} continues to hold when $v_0,V\in C^{1,\nu}(\overline{\Omega})$ for some $\nu>0$, but for our purposes it is enough to work in the smooth setting.
\end{remark}

\begin{proof}
    Let us start by recalling that if $f\in C^{1,\alpha}(\partial\Omega)$ satisfies $\|f\|_{C^{1,\alpha}(\partial\Omega)}\leq C_0$ and $u$ is a bounded weak solution of 
    \begin{equation}
    \label{eq: Lieberman}
    \begin{cases}
        \Div(|\nabla u|^{p-2}\nabla u)=0& \text{ in }\Omega,\\
        u=f&\text{ on }\partial\Omega
    \end{cases}
    \end{equation}
    with $\|u\|_{L^{\infty}(\Omega)}\leq C_1$, then there exist positive constants $\beta'=\beta'(\alpha,p,n)\in (0,1)$, $C=C(\alpha, p,C_0,C_1,n,\Omega)>0$ such that
    \begin{equation}
    \label{eq: uniform estimate Lieberman}
        \|u\|_{C^{1,\beta'}(\overline{\Omega})}\leq C.
    \end{equation}
    This is a special case of \cite[Theorem 1]{Lieberman1988} with $m=p-2$, $\kappa=0$, $\lambda=\min(1,p-1)$ and $\Lambda=\max(1,p-1)$ (see~\eqref{uniform ellipticity of J}).
    
 Now, suppose that $v_0\in C^\infty(\overline{\Omega})$ is a smooth solution of \eqref{eq-p-Laplace} without critical points and $\phi\in C^\infty(\partial\Omega)$ is a given boundary value. Furthermore, let $v_\tau\in W^{1,p}(\Omega)$, $-1<\tau<1$, be the unique solution to 
\begin{equation}\label{eq-tau}
\begin{cases}
\Div(|\nabla v_\tau|^{p-2}\nabla v_\tau)=0&\text{ in }\Omega,\\
 v_\tau=v_0+\tau\phi&\text{ on }\partial\Omega.
\end{cases}
\end{equation}
Note that by Theorem \ref{thm: max principle} with $p=q$, $a=0$ we have
\begin{equation}
\label{eq: L infty bound on v tau}
    \|v_\tau\|_{L^{\infty}(\Omega)}\leq \|v_0\|_{L^{\infty}(\Omega)}+\|\Phi\|_{L^{\infty}(\partial\Omega)},
\end{equation}
where $\Phi\in C^{\infty}(\overline{\Omega})$ is any smooth extension of $\phi\in C^{\infty}(\partial\Omega)$. Since the boundary data $v_0+\tau\phi$ satisfy $\|v_0+\tau\phi\|_{C^{1,\alpha}(\partial\Omega)}\leq \|v_0\|_{C^{1,\alpha}(\partial\Omega)}+\|\phi\|_{C^{1,\alpha}(\partial\Omega)}$, the solutions $v_\tau$ are uniformly bounded with respect to $\tau$ in $C^{1,\beta'}(\overline{\Omega})$ by the aforementioned result of Lieberman. Next, fix some $\gamma'<\delta'<\beta'$. By the theorem of Arzel\'a--Ascoli, there must exist $\widetilde{v}_0\in C^{1,\beta'}(\overline{\Omega})$ such that, after passing to a subsequence, $v_\tau\to \widetilde{v}_0$ in $C^{1,\delta'}(\overline{\Omega})$ as $\tau\to 0$. Taking the limit of \eqref{eq-tau} (e.g. in the sense of $\mathscr{D}'(\Omega)$) and using the trace theorem, it follows that $\widetilde{v}_0$ solves \eqref{eq-p-Laplace}, with the same Dirichlet data as $v_0$. Therefore, by Theorem \ref{thm: well-posedness double phase}, we may deduce that $\widetilde{v}_0=v_0$ and hence $v_\tau\to v_0$ in $C^{1,\delta'}(\overline{\Omega})$ as $\tau\to 0$. This also shows that $v_\tau$ has no critical points for $|\tau|\leq \tau_0$, when $\tau_0>0$ is small enough. 

Let $V_\tau\in C^{1,\beta'}(\overline{\Omega})$ be defined by the Ansatz
\begin{equation}
\label{eq: expansion v tau final proof}
v_\tau=v_0+\tau V_\tau.
\end{equation}
Then, similarly as in the proof of Proposition \ref{prop: asymptotic expansion}, we see that there holds
\begin{equation}
\label{eq: expansion p laplacian}
\Div(|\nabla v_\tau|^{p-2}\nabla v_\tau)=\Div(|\nabla v_0|^{p-2}\nabla v_0)+\tau\Div(\mathcal{A}^\tau\nabla V_\tau),
\end{equation}
where
\begin{equation}
\mathcal{A}^\tau=\int_0^1\nabla_\xi J^p(\nabla v_0+t\tau\nabla V_\tau)\,dt.
\end{equation}
Therefore, as $v_\tau$ and $v_0$ are $p$-harmonic, $V_\tau$ solves
\begin{equation}
\label{eq: V tau equation}
\begin{cases}
    \Div(\mathcal{A}^\tau\nabla V_\tau)=0& \text{ in }\Omega,\\
    V_\tau=\phi&\text{ on }\partial\Omega.
\end{cases}
\end{equation}
Since $v_0$ does not have critical points and $\tau V_\tau\to 0$ in $C^{1,\delta'}(\overline{\Omega})$ as $\tau\to 0$, it follows (by Claim \ref{claim: properties of A eps V} and Remark \ref{rem: uniform bound for small C 1 beta functions}, with $a\equiv0$) that there exists a $\tau_0>0$ such that the matrices $\mathcal{A}^\tau\in C^{0,\delta'}(\overline{\Omega})$ are uniformly elliptic and furthermore they have ellipticity and $\|\cdot\|_{C^{0,\delta'}(\overline{\Omega})}$ bounds that are independent of $\tau\in[-\tau_0,\tau_0]$. From \cite[Theorems 8.33 \& 8.34]{GiTru} we deduce that $V_\tau\in C^{1,\beta'}(\overline{\Omega})$ satisfies
\begin{equation}
\label{eq: preliminary Holder estimate V tau}
    \|V_\tau\|_{C^{1,\delta'}(\overline{\Omega})}\leq C(\|V_\tau\|_{L^{\infty}(\Omega)}+\|\Phi\|_{C^{1,\delta'}(\overline{\Omega})}),
\end{equation}
where $C>0$ only depends on the uniform the ellipticity constant of $\mathcal{A}_\tau$, a uniform upper bound of $\|\mathcal{A}_\tau\|_{C^{0,\delta'}(\overline{\Omega})}$, $\Omega$ and $n$. Hence, the maximum principle for (uniformly) elliptic PDEs \cite[Proposition 2.30]{Elliptic-PDEs-Xavi} ensures that 
\begin{equation}
\label{eq: L infinity estimate V tau}
    \|V_\tau\|_{L^{\infty}(\Omega)}\leq \|\phi\|_{L^{\infty}(\partial\Omega)}
\end{equation}
and thus $V_\tau$ is uniformly bounded with respect to $\tau\in[-\tau_0,\tau_0]$ in $C^{1,\delta'}(\overline{\Omega})$.

By the theorem of Arzel\'a--Ascoli, there must exist a $V\in C^{1,\delta'}(\overline{\Omega})$ such that $V_\tau\to V$ in $C^{1,\gamma'}(\overline{\Omega})$ as $\tau\to 0$, after passing to a subsequence. Passing to the limit in \eqref{eq: V tau equation} shows that $V$ must be a weak solution to \eqref{eq-V}. For the previous convergence assertion, we use the fact that $V\mapsto \nabla_\xi J^p(\nabla v_0+t\nabla V)$ is continuous from $C^1(\overline{\Omega})$ to $C(\overline{\Omega})$ (when $\|V\|_{C^{1,\gamma'}(\overline{\Omega})}\leq \lambda_0/2$, see Remark \ref{rem: uniform bound for small C 1 beta functions}), the uniform bound of $\nabla_\xi J^p(\nabla v_0+t\tau\nabla V_\tau)$ in $t$ provided by Claim \ref{claim: properties of A eps V} and Remark \ref{rem: uniform bound for small C 1 beta functions} for small $\tau$ as well as the dominated convergence theorem. Since solutions to \eqref{eq-V} are unique, it follows that 
\begin{equation}
\label{eq: convergence of V tau}
    V_\tau\to V\text{ in }C^{1,\gamma'}(\overline{\Omega})
\end{equation}
as $\tau\to 0$, without passing to a subsequence. Hence, we deduce that
\[
    v_\tau=v_0+\tau V+\tau (V_{\tau}-V)=  v_0+\tau V+o(\tau)
\]
in $C^{1,\gamma'}(\overline{\Omega})$ as $\tau\to 0$. This concludes the proof.
\end{proof}

\section{Proof of Theorem \ref{thm: main theorem double phase}}
\label{sec: proof of thm 1.1}

In this section, we establish the proof of Theorem \ref{thm: main theorem double phase}. In Section \ref{subsec: Inverse double phase problem for p smaller q}, we prove Theorem \ref{thm: main theorem double phase} in the range $p<q$ and in Section \ref{subsec: Inverse double phase problem for p bigger q} in the range $p>q$. 

\subsection{Reconstruction of \texorpdfstring{$a$}{a} in the range \texorpdfstring{$p<q$}{p smaller q}}
\label{subsec: Inverse double phase problem for p smaller q}

In this section, we detail our reconstruction procedure of the coefficient $a$ in the double phase problem \eqref{eq: PDE double phase} when $p<q$.

\begin{proof}[Proof of Theorem \ref{thm: main theorem double phase} for $p<q$]

Let $u_\eps\in W^{1,p,q}_a(\Omega)$, $\eps>0$, be the unique solutions to \eqref{eq: epsilon PDE} for some $p$-harmonic function $v\in C^{\infty}(\overline{\Omega})$ without critical points and let $0<\gamma<\beta<1$ be some exponents as in Proposition \ref{prop: asymptotic expansion}, which we shall fix at a later stage of the proof. We claim that for any $g\in W^{1-1/q,q}(\partial\Omega)$, with extension $\omega\in W^{1,q}(\Omega)$, there holds
\begin{equation}
\label{eq: DN map expansion}
\begin{split}
\langle \Lambda_a\eps v, g\rangle&=\eps^{p-1} \int_{\Omega} |\nabla v|^{p-2}\nabla v\cdot\nabla \omega\,dx\\
&\quad +\eps^{q-1}\int_{\Omega} (A^p_v\nabla R_v +a|\nabla v|^{q-2}\nabla v)\cdot\nabla \omega\,dx
+o(\eps^{q-1})
\end{split}
\end{equation}
as $\eps\to 0$, where $R_v\in C^{2,\alpha}(\overline{\Omega})$ is the unique solution of \eqref{eq: PDE for R0} (see~Proposition \ref{prop: asymptotic expansion}).

Indeed, recalling from the proof of Proposition \ref{prop: asymptotic expansion} that we may expand $u_\eps$ as
\begin{equation}
\label{eq: asymptotic expansion proof thm 1}
    u_\eps=\eps (v+w_{v,\eps})\text{ with }w_{v,\eps}=\eps^{q-p}R_{v,\eps}\text{ and } R_{v,\eps}\to R_v\text{ in }C^{2,\gamma}(\overline{\Omega}),
\end{equation}
we obtain by formula \eqref{Taylor} the identities
\begin{equation}
\label{eq: expansion p phase thm 1}
\begin{split}
|\nabla u_\eps|^{p-2}\nabla u_\eps&=\eps^{p-1}|\nabla v|^{p-2}\nabla v+\eps^{q-1} A^p_v(w_{v,\eps})\nabla R_{v,\eps}\\
&=\eps^{p-1}|\nabla v|^{p-2}\nabla v+\eps^{q-1} A^p_v\nabla R_v \\
&\quad +\eps^{q-1}\left[A^p_v\nabla (R_{v,\eps}-R_v)+(A^p_v(w_{v,\eps})-A^p_v)\nabla R_{v,\eps}  \right]
\end{split}
\end{equation}
and
\begin{equation}
\label{eq: expansion p,q phase thm 1}
a|\nabla u_\eps|^{q-2}\nabla u_\eps=\eps^{q-1}a|\nabla v|^{q-2}\nabla v+\eps^{q-1}\eps^{q-p} aA^q_v(w_{v,\eps})\nabla R_{v,\eps}.
\end{equation}
The matrices $A^r_v(w_{v,\eps})$, $1<r<\infty$, appearing in \eqref{eq: expansion p phase thm 1}--\eqref{eq: expansion p,q phase thm 1}, are given by \eqref{eq: coeff p and p,q phase} and we added the subscript $v$ on these matrices and the functions $w_\eps,R_\eps$ to highlight their dependence on $v$. 
Since $R_{v,\eps}\to R_v$ and $w_{v,\eps}\to 0$ in $C^{2,\gamma}(\overline{\Omega})$ as $\eps\to 0$, the claim \eqref{eq: DN map expansion} follows from the definition of the DN map \eqref{eq: rigorous DN map} and the dominated convergence theorem. This shows that, for any $p$-harmonic function $v$ without critical points and any $g\in W^{1-1/q,q}(\partial\Omega)$ with extension $\omega\in W^{1,q}(\Omega)$, we can compute the quantity
\begin{equation}
\label{eq: def I}
\begin{split}
I(v,g)&\vcentcolon =\lim_{\eps\to0^+}\eps^{1-q}( \langle \Lambda_a\eps v, g\rangle- \langle\Lambda_0\eps v,g\rangle )\\
&=\int_{\Omega} (A^p_v\nabla R_v +a|\nabla v|^{q-2}\nabla v)\cdot\nabla \omega\,dx,
\end{split}
\end{equation}
from the DN map $\Lambda_a$ of the double phase problem and the DN map $\Lambda_0$ of the $p$-Laplacian $\Delta_p$.

Next, we consider the family $(v_\tau)_{\tau\in [-\tau_0,\tau_0]}$ of $p$-harmonic functions constructed in Lemma \ref{lemma: family of p harmonic functions}, which are supposed to have zeroth order term $v_0$ and first order term $V$. Thus, the previous calculation can be applied to $v_\tau$ when $\beta\vcentcolon =\beta'$ and $\gamma\vcentcolon =\gamma'<\beta'$. We wish to compute the derivative
\begin{equation}
\label{eq: def J}
    J(v_0,V,\omega)\vcentcolon =\left.\frac{d}{d\tau}\right|_{\tau=0} I(v_\tau,g)=\lim_{\tau\to 0}\frac{I(v_\tau,g)-I(v_0,g)}{\tau}.
\end{equation}
First, note that we have
\begin{equation}
\label{eq: A dot}
\begin{split}
\dot A_{v_0}^p (V)&\vcentcolon =\lim_{\tau\to0}\tau^{-1}\left(A^p_{v_\tau}-A^p_{v_0}\right)\\
&=\lim_{\tau\to0}\tau^{-1}\left(\nabla_\xi J^p(\nabla v_\tau)-\nabla_\xi J^p(\nabla v_0)\right)\\
&=(p-2)|\nabla v_0|^{p-4}(\nabla v_0\cdot\nabla V)\Big[\mathbf{1}+(p-4)\frac{\nabla v_0\otimes\nabla v_0}{|\nabla v_0|^2}\Big]\\
&\quad +(p-2)|\nabla v_0|^{p-4}( \nabla v_0\otimes\nabla V+\nabla V\otimes\nabla v_0)
\end{split}
\end{equation}
in $C^{0,\gamma}(\overline{\Omega})$. This can be seen as follows. First, using \eqref{eq: derivatives of Jj} and a straight forward calculation, we obtain
\begin{equation}
\label{eq: second order derivatives J}
    \partial_{\xi_i\xi_j}J^p_k (\xi)=(p-2)|\xi|^{p-4}(\xi_i \delta_{kj}+\xi_j\delta_{ki}+\xi_k\delta_{ij})+(p-2)(p-4)|\xi|^{p-6}\xi_i\xi_j\xi_k
\end{equation}
for all $1\leq i,j,k\leq n$ and $\xi\neq 0$. Secondly, the map $\xi\mapsto \nabla_\xi J^p(\xi)$ is locally Lipschitz away from the origin $\xi=0$ (see~\eqref{eq: estimate mikko}). Hence, the fundamental theorem of calculus ensures
\begin{equation}
\label{eq: FTC first der}
    \partial_{\xi_i}J^p_k(\eta)-\partial_{\xi_i}J^p_k(\xi)=\int_0^1 \partial_{\xi_i\xi_j}J^p_k(\xi+t(\eta-\xi))\,dt\,(\eta_j-\xi_j)
\end{equation}
for all $1\leq j,k\leq n$, as long as the straight line from $\xi$ to $\eta$ does not go through the origin. Thus, the nondegeneracy of $\nabla v_\tau$, $|\tau|\leq \tau_0$, ensures that we can apply \eqref{eq: FTC first der} to $\eta=\nabla v_\tau$ and $\xi=\nabla v_0$. Using the convergence $V_\tau\to V$ in $C^{1,\gamma}(\overline{\Omega})$ as $\tau\to 0$ and \eqref{eq: second order derivatives J} in the obtained identity, we get the formula in \eqref{eq: A dot}.

Furthermore, by \cite[Theorem 8.34]{GiTru} and $\dot A_{v_0}^p (V)\in C^{0,\gamma}(\overline{\Omega})$, we can denote by $\dot{R}\in C^{1,\gamma}(\overline{\Omega})$ the unique solution to
\begin{equation}
\label{eq: PDE for R dot}
    \begin{cases}
      \Div(A^p_{v_0}\nabla\dot R)=-\Div(\dot A_{v_0}^p(V)\nabla R_{v_0}  )-\Div(a\, A^q_{v_0}\nabla V)& \text{ in }\Omega,\\
       \dot R=0&\text{ on }\partial\Omega.
    \end{cases}
\end{equation}
Next, let us observe that using \eqref{eq: PDE for R0} we can write
\begin{equation}
\label{eq: PDE for R v tau minus R v 0}
\begin{cases}
\Div(A^p_{v_0}\nabla (R_{v_\tau}-R_{v_0}))=\Div((A^p_{v_0}-A^p_{v_\tau})\nabla R_{v_\tau}  )+\Div f_\tau& \text{ in }\Omega,\\
 R_{v_\tau}-R_{v_0}=0&\text{ on }\partial\Omega,
\end{cases}
\end{equation}
where
\begin{equation}
\label{eq: def of f tau}
    f_\tau\vcentcolon = a (|\nabla v_0|^{q-2}\nabla v_0-|\nabla v_\tau|^{q-2}\nabla v_\tau)\in C^{0,\gamma}(\overline{\Omega}).
\end{equation}
Next, we claim that there holds 
\begin{equation}
\label{eq: Holder estimate R v tau}
    \sup_{|\tau|\leq\tau_0}\|\nabla R_{v_\tau}\|_{C^{0,\gamma}(\overline{\Omega})}<\infty
\end{equation}
for some possibly smaller $\tau_0>0$. First, we observe that \eqref{eq: PDE for R0}, Claim \ref{claim: properties of A eps V} applied to $A^p_{v_\tau}$ and \cite[Theorem 8.33]{GiTru} imply
\begin{equation}
\label{eq: preliminary holder estimate for R v tau}
    \|R_{v_\tau}\|_{C^{1,\gamma}(\overline{\Omega})}\leq C(\|R_{v_\tau}\|_{L^{\infty}(\Omega)}+\|a|\nabla v_\tau|^{q-2}\nabla v_\tau\|_{C^{0,\gamma}(\overline{\Omega})}),
\end{equation}
where $C>0$ is independent of $\tau$. From the proof of Lemma \ref{lemma: family of p harmonic functions}, we know that $v_\tau=v_0+\tau V_\tau$ and $V_\tau\in C^{1,\gamma}(\overline{\Omega})$ satisfies the estimate
\begin{equation}
    \|V_\tau\|_{C^{1,\gamma}(\overline{\Omega})}\leq C\|\Phi\|_{C^{1,\beta}(\overline{\Omega})},
\end{equation}
where $C>0$ is independent of $\tau$ and $\Phi\in C^{\infty}(\overline{\Omega})$ is an extension of $\phi\in C^{\infty}(\partial\Omega)$ (see~\eqref{eq: preliminary Holder estimate V tau} and \eqref{eq: L infinity estimate V tau}). This ensures the bound
\begin{equation}
\label{eq: holder estimate for v tau}
    \|v_\tau\|_{C^{1,\gamma}(\overline{\Omega})}\leq C(\|v_0\|_{C^{1,\gamma}(\overline{\Omega})}+\|\Phi\|_{C^{1,\beta}(\overline{\Omega})})
\end{equation}
uniformly in $\tau$. Furthermore, \cite[Theorem 8.16]{GiTru} and \eqref{eq: PDE for R0} (for $R_{v_\tau}$) demonstrate that
\begin{equation}
\label{eq: L infinity estimate for R v tau}
    \|R_{v_\tau}\|_{L^{\infty}(\Omega)}\leq C\|a|\nabla v_\tau|^{q-2}\nabla v_\tau\|_{L^{\infty}(\Omega)}.
\end{equation}
Combining this with \eqref{eq: preliminary holder estimate for R v tau}, we get
\begin{equation}
\label{eq: preliminary holder estimate 2 for R v tau}
    \|R_{v_\tau}\|_{C^{1,\gamma}(\overline{\Omega})}\leq C\|a|\nabla v_\tau|^{q-2}\nabla v_\tau\|_{C^{0,\gamma}(\overline{\Omega})}
\end{equation}
uniformly in $\tau$. By the product rule \eqref{eq: product rule Holder}, the uniform nondegeneracy of $\nabla v_\tau$, \eqref{eq: holder estimate for v tau} and  Lemma \ref{auxiliary lemma 2}, we can uniformly bound the right hand side of \eqref{eq: preliminary holder estimate 2 for R v tau}, which yields \eqref{eq: Holder estimate R v tau}.

We next assert that
\begin{equation}
\label{eq: C 1 delta convergence of R v tau}
    R_{v_\tau}\to R_{v_0}\text{ in } C^{1,\delta}(\overline{\Omega})
\end{equation}
for some $0<\delta<\gamma$. In fact, from \eqref{eq: Holder estimate R v tau} and a compactness argument we can conclude that $R_{v_\tau}\to\widetilde{R}_{v_0}$ in $C^{1,\delta}(\overline{\Omega})$ along a subsequence for some $0<\delta<\gamma$. Using \eqref{eq: v tau expansion} and passing to the limit in the PDE for $R_{v_\tau}$, we see that $R_{v_0}$ and $\widetilde{R}_{v_0}$ solve the same Dirichlet problem, and hence $R_{v_0}=\widetilde{R}_{v_0}$. That the coefficient $A^p_{v_\tau}$ in the PDE for $R_{v_\tau}$ converges to $A^p_{v_0}$ can be shown in a similar way as done at the end of the proof of Lemma \ref{lemma: family of p harmonic functions}. Since the limit is independent of the chosen subsequence, we infer the convergence \eqref{eq: C 1 delta convergence of R v tau}. 

Finally, we show that
\begin{equation}
\label{eq: R dot equal derivative}
    \dot R=\lim_{\tau\to 0}\frac{R_{v_\tau}-R_{v_0}}{\tau}
\end{equation}
in $C^{1,\kappa}(\overline{\Omega})$ for some $0<\kappa<\delta$, where $\dot{R}\in C^{1,\gamma}(\overline{\Omega})$ is the unique solution of \eqref{eq: PDE for R dot}. From \eqref{eq: PDE for R v tau minus R v 0}, we see that
\begin{equation}
\label{eq: PDE difference quotient}
    \begin{cases}
\Div\left(A^p_{v_0}\nabla \frac{R_{v_\tau}-R_{v_0}}{\tau}\right)=\Div\left(\frac{A^p_{v_0}-A^p_{v_\tau}}{\tau}\nabla R_{v_\tau}  \right)+\Div \frac{f_\tau}{\tau}& \text{ in }\Omega,\\
 \frac{R_{v_\tau}-R_{v_0}}{\tau}=0&\text{ on }\partial\Omega,
\end{cases}
\end{equation}
for any $\tau\neq 0$. From \eqref{eq: A dot} and \eqref{eq: C 1 delta convergence of R v tau}, we deduce that
\begin{equation}
\label{eq: boundedness of first term PDE for diff quotient}
\sup_{|\tau|\leq\tau_0}\max\left(\left\|\frac{A^p_{v_\tau}-A^p_{v_0}}{\tau}\right\|_{C^{0,\gamma}(\overline{\Omega})}, \|R_{v_{\tau}}\|_{C^{1,\delta}(\overline{\Omega})}\right)<\infty
\end{equation}
for some small $\tau_0>0$. On the other hand, the formulas \eqref{Taylor}, \eqref{eq: def of f tau} and \eqref{eq: expansion v tau final proof} ensure that there holds
\[
    f_\tau=-\tau a\left(\int_0^1\nabla_\xi J^q(\nabla v_0+t\tau \nabla V_\tau)\,dt\right)\nabla V_\tau.
\]
So, from \eqref{eq: convergence of V tau} and the product rule in H\"older spaces, we see that
\begin{equation}
\label{eq: convergence of second term PDE diff quotient}
    f_\tau/\tau\to -a\, A^q_{v_0}\nabla V
\end{equation}
in $C^{0,\gamma}(\overline{\Omega})$ as $\tau\to 0$, which demonstrates that there holds
\begin{equation}
\label{eq: boundedness of second term PDE diff quotient}
    \left\|f_\tau/\tau\right\|_{C^{0,\gamma}(\overline{\Omega})}<\infty
\end{equation}
for $|\tau|\leq \tau_0$ for some small $\tau_0>0$. More concretely, to obtain the convergence \eqref{eq: convergence of second term PDE diff quotient} we make use of the fact that
\begin{equation}
\label{eq: Holder convergence of nabla J}
    \|\nabla_\xi J^q(\nabla v_0+t\tau\nabla V_\tau)-\nabla_\xi J^q(\nabla v_0)\|_{C^{0,\gamma}(\overline{\Omega})}\to 0
\end{equation}
as $\tau\to 0$, uniformly for $0\leq t\leq 1$. This in turn can be shown by using \eqref{eq: FTC first der}, \eqref{eq: second order derivatives J} and Lemma \ref{auxiliary lemma 2}. For similar arguments, we refer the reader to the proof of Claim \ref{claim: properties of A eps V} and Lemma \ref{lemma: family of p harmonic functions}. Now, using \eqref{eq: product rule Holder}, \eqref{eq: boundedness of first term PDE for diff quotient}, \eqref{eq: boundedness of second term PDE diff quotient} and \cite[Theorem 8.16 \& 8.33]{GiTru}, we see that 
\[
    \left\|\frac{R_{v_\tau}-R_{v_0}}{\tau}\right\|_{C^{1,\delta}(\overline{\Omega})}<\infty
\]
uniformly in $\tau$. Hence, there exists $\dot{\mathcal{R}}\in C^{1,\delta}(\overline{\Omega})$ such that 
\begin{equation}
\label{eq: conv diff quotient R v tau}
    \frac{R_{v_\tau}-R_{v_0}}{\tau}\to \dot{\mathcal{R}}
\end{equation}
in $C^{1,\kappa}(\overline{\Omega})$ for some $0<\kappa<\delta$. By the convergence results \eqref{eq: A dot}, \eqref{eq: convergence of second term PDE diff quotient} and \eqref{eq: conv diff quotient R v tau}, we see that passing to the limit $\tau\to 0$ in \eqref{eq: PDE difference quotient} reveals that $\dot{\mathcal{R}}\in C^{1,\delta}(\overline{\Omega})$ solves \eqref{eq: PDE for R dot}. Thus, we may conclude that $\dot{\mathcal{R}}=\dot{R}$, which in turn implies the convergence \eqref{eq: R dot equal derivative}.

Finally, we observe that \eqref{eq: expansion v tau final proof}, \eqref{Taylor}, $V_\tau\to V$ in $C^{1,\gamma}(\overline{\Omega})$ and \eqref{eq: Holder convergence of nabla J} demonstrate that there holds
\begin{equation}
\label{eq: derivative second term in I}
\begin{split}
&\left.\frac{d}{d\tau}\right|_{\tau=0}\int_{\Omega}a|\nabla v_\tau|^{q-2}\nabla v_\tau\cdot\nabla \omega\,dx\\
&=\lim_{\tau\to 0}\int_\Omega a\left(\int_0^1\nabla_\xi J^q(\nabla v_0+t\tau\nabla V_\tau)dt\right)\,\nabla V_\tau\cdot\nabla \omega\,dx\\
&=\lim_{\tau\to 0}\int_\Omega a\left(\int_0^1[\nabla_\xi J^q(\nabla v_0+t\tau\nabla V_\tau)-\nabla_\xi J^q(\nabla v_0)]\,dt\right)\,\nabla V_\tau\cdot\nabla \omega\,dx\\
&\quad +\lim_{\tau\to 0}\int_\Omega a\nabla_\xi J^q(\nabla v_0)\nabla V_\tau\cdot\nabla \omega\,dx\\
&=\int_{\Omega}a \,A_{v_0}^q\nabla V\cdot \nabla\omega\,dx.
\end{split}
\end{equation}

It follows from \eqref{eq: def J}, \eqref{eq: def I} with $v=v_\tau$, \eqref{eq: A dot}, \eqref{eq: C 1 delta convergence of R v tau}, \eqref{eq: R dot equal derivative} and  \eqref{eq: derivative second term in I} that we have
\begin{equation}
\label{eq: explicit expression of J}
\begin{split}
J(v_0,V,\omega)&=\lim_{\tau\to 0}\int_\Omega \frac{A^p_{v_\tau}\nabla R_{v_\tau}-A^p_{v_0}\nabla R_{v_0}}{\tau}\cdot\nabla \omega\,dx\\
&\quad +\left.\frac{d}{d\tau}\right|_{\tau=0}\int_{\Omega}a|\nabla v_\tau|^{q-2}\nabla v_\tau\cdot\nabla \omega\,dx\\
&=\lim_{\tau\to 0}\int_\Omega \left(\frac{A^p_{v_\tau}-A^p_{v_0}}{\tau}\nabla R_{v_\tau}+A^p_{v_0}\frac{\nabla R_{v_\tau}-\nabla R_{v_0}}{\tau}\right)\cdot\nabla \omega\,dx\\
&\quad +\int_{\Omega}a \,A_{v_0}^q\nabla V\cdot \nabla\omega\,dx\\
&= \int_{\Omega}\left(\dot A_{v_0}^p(V)\nabla R_{v_0} +A_{v_0}^p\nabla\dot R+a A^q_{v_0}\nabla V \right)\cdot
\nabla\omega\,dx.
\end{split}
\end{equation}

Suppose that $V_1, V_2\in C^2(\overline{\Omega})$ are two solutions of the equation $\Div(A_{v_0}^p\nabla V)=0$ with boundary data $\phi_1,\phi_2\in C^{\infty}(\partial\Omega)$ (see~\cite[Theorem 6.15]{GiTru}). Using Lemma \ref{lemma: family of p harmonic functions}, the symmetry of the matrices $A^p_{v_0}, \dot{A}^p_{v_0}(V_1)\in C^1(\overline{\Omega})$ (the regularity assertions follows from \eqref{eq: derivatives of Jj}, \eqref{eq: second derivative of J}, \eqref{eq: A dot}, \eqref{eq: derivative of mod power s}, the nondegeneracy of $\nabla v_0$ and $v_0,V_1\in C^2(\overline{\Omega})$) and the fact that $R,\dot{R}\in C^{1,\kappa}(\overline{\Omega})$ satisfy $R_{v_0}|_{\partial\Omega}=\dot{R}|_{\partial\Omega}=0$, we get by an integration by parts in the first two terms of \eqref{eq: explicit expression of J} the identity
\begin{equation}
\label{eq: J for two different solutions}
\begin{split}
    &J(v_0,V_1,V_2)=\int_{\Omega}\left(\dot A_{v_0}^p(V_1)\nabla R_{v_0} +A_{v_0}^p\nabla\dot R+a A^q_{v_0}\nabla V_1 \right)\cdot
\nabla V_2\,dx\\
    &\quad =-\int_{\Omega}[R_{v_0}\Div(\dot A_{v_0}^p(V_1)\nabla V_2)+\dot{R}\Div(A^p_{v_0}\nabla V_2)]\,dx+\int_{\Omega}a\,A^q_{v_0}\nabla V_1\cdot\nabla V_2\,dx\\ 
    &\quad =\int_{\Omega}\left(a \, A_{v_0}^q\nabla V_1\cdot\nabla V_2 -R_{v_0}\Div(\dot A_{v_0}^p(V_1)\nabla V_2)   \right)\,dx.
\end{split}
\end{equation}
Up to now we have assumed that all functions are real-valued, but in the next final step we want to choose $V_j$, $j=1,2$, complex-valued while keeping $v_0$ real-valued. For this let us note that $V\mapsto \dot{A}^p_{v_0}(V)$ is linear and hence if we decompose the possibly complex-valued solutions $V_j$ as $V_j=U_j+i W_j$, $j=1,2$, then we have
\[
    \begin{split}
        &\int_{\Omega}\left(a \, A_{v_0}^q\nabla V_1\cdot\nabla V_2 -R_{v_0}\Div(\dot A_{v_0}^p(V_1)\nabla V_2)   \right)\,dx\\
        &=\left(J(v_0,U_1,U_2)-J(v_0,W_1,W_2)\right)+i\left(J(v_0,U_1,W_2)+J(v_0,W_1,U_2)\right).
    \end{split}
\]
As all terms in the last line are determined by the DN map $\langle\Lambda_a f,g\rangle$ for appropriate real-valued boundary conditions $f$ and $g$, the same remains true for the integral
\[
    \int_{\Omega}\left(a \, A_{v_0}^q\nabla V_1\cdot\nabla V_2 -R_{v_0}\Div(\dot A_{v_0}^p(V_1)\nabla V_2)   \right)\,dx,
\]
which we still denote by $J(v_0,V_1,V_2)$. This notation is justified by the fact that if $V_j$, $j=1,2$, solve \eqref{eq-V}, then the same holds for the real and imaginary parts as $v_0$ is real-valued.

Let $z\in\R^n$, $|z|=1$ and set $v_0=z\cdot x$. Then using $A^r_{v_0}=\nabla_\xi J^r(\nabla v_0)$, \eqref{eq: derivatives of Jj} and \eqref{eq: A dot}, we deduce
\begin{equation}
\label{eq: A p z}
A^p_{v_0}=\mathbf{1}+(p-2)z\otimes z, \quad A^q_{v_0}=\mathbf{1}+(q-2)z\otimes z,
\end{equation}
and
\begin{equation}
\label{eq: A dot z}
\dot A^p_{v_0}(V_1)=(p-2)\Big[(z\cdot\nabla V_1)\mathbf{1}+(p-4)(z\cdot\nabla V_1) z\otimes z +z\otimes\nabla V_1+\nabla V_1\otimes z  \Big].
\end{equation}

 Let $\xi\in\R^n\setminus\{0\}$ be such that $\xi\perp z$. For $s>0$ to be chosen later, we introduce the complex vectors
\begin{equation}
\zeta_\pm=\pm s z+i\xi.
\end{equation}
By \eqref{eq: A p z}, $\zeta_{\pm}\cdot\zeta_{\pm}=s^2-|\xi|^2$, $z\cdot \zeta_{\pm}=\pm s$ and 
\begin{equation}
\label{eq: tensor times vector}
    (a\otimes b) c=a(b\cdot c)
\end{equation}
for all $a,b,c\in\C^n$, we have
\begin{equation}
\zeta_\pm\cdot A_{v_0}^p\zeta_\pm=(p-1)s^2-|\xi|^2.
\end{equation}
So, if we choose 
\begin{equation}
\label{eq: choice of s}
    s=(p-1)^{-\frac{1}{2}}|\xi|,
\end{equation}
then the functions
\begin{equation}
V_1(x)=e^{\zeta_+\cdot x},\quad V_2(x)=e^{\zeta_-\cdot x},
\end{equation}
are both global smooth solutions of $\Div(A_{v_0}^p\nabla V)=0$. 

Using \eqref{eq: A dot z}, \eqref{eq: tensor times vector}, $\zeta_+\cdot \zeta_{-}=-(s^2+|\xi|^2)$ and \eqref{eq: choice of s}, we get
\begin{equation}
\begin{split}
\dot A_{v_0}^p(V_1)\nabla V_2
&=(p-2)e^{2i\xi\cdot x}\left(s\mathbf{1}+s(p-4)z\otimes z+z\otimes\zeta_++\zeta_+\otimes z  \right)\zeta_-\\
&=(p-2)e^{2i\xi\cdot x}\left(s\zeta_--s^2(p-4)z-(s^2+|\xi|^2)z-s\zeta_+   \right)\\
&=-(p-2)e^{2i\xi\cdot x}[(p-1)s^2+|\xi|^2]z=-2(p-2)|\xi|^2 e^{2i\xi\cdot x}z
\end{split}
\end{equation}
and thus it follows from $\xi\perp z$ that
\begin{equation}
\label{eq: dot PDE is zero}
\Div(\dot A_{v_0}^p(V_1)\nabla V_2)=0.
\end{equation}
On the other hand, by \eqref{eq: tensor times vector}, \eqref{eq: choice of s}, $\zeta_{\pm}\cdot\zeta_{\pm}=s^2-|\xi|^2$, $z\cdot \zeta_{\pm}=\pm s$ and $\zeta_+\cdot \zeta_{-}=-(s^2+|\xi|^2)$ we obtain
\begin{equation}
\label{eq: gradient mixed terms}
\begin{split}
    A^q_{v_0}\nabla V_1\cdot \nabla V_2&= e^{2i\xi\cdot x} A^q_{v_0}\zeta_+\cdot \zeta_-\\
    &=-e^{2i\xi\cdot x}((q-1)s^2+|\xi|^2)\\
    &=-\frac{p+q-2}{p-1}|\xi|^2e^{2i\xi\cdot x}.
\end{split}
\end{equation}
Therefore, by using \eqref{eq: J for two different solutions}, \eqref{eq: dot PDE is zero} and \eqref{eq: gradient mixed terms}, we deduce that
\begin{equation}
\begin{split}
    J(z\cdot x, e^{\zeta_+\cdot x},e^{\zeta_-\cdot x})&=-\frac{p+q-2}{4(p-1)}\int_{\Omega}a(x)|2\xi|^2e^{2i\xi\cdot x}\,dx.
\end{split}
\end{equation}
Hence, if we set $a$ outside of $\overline{\Omega}$ equal to zero, then we get
\begin{equation}
\label{eq: fourier of a}
    \hat{a}(\xi)=-\frac{4(p-1)}{p+q-2}\frac{J(z\cdot x, e^{\widetilde{\zeta}_+\cdot x},e^{\widetilde{\zeta}_-\cdot x})}{|\xi|^2}
\end{equation}
for all $\xi\perp z$, where 
\[
    \widetilde{\zeta}_{\pm}=\pm \frac{|\xi|}{2(p-1)^{1/2}} z+i\frac{\xi}{2}.
\]
As formula \eqref{eq: fourier of a} holds for all $z\in \mathbb{S}^{n-1}$ and $\xi\neq 0$ with $\xi\perp z$, we can conclude that $\hat{a}(\xi)$ is determined by $J$ for all $\xi\in\R^n$. Therefore, by Fourier's inversion theorem we have shown that $a$ can be recovered from the DN map $\Lambda_a$ (see~\eqref{eq: def I}, \eqref{eq: def J} and \eqref{eq: fourier of a}).

This concludes the proof.
\end{proof}

\subsection{Reconstruction of \texorpdfstring{$a$}{a} in the range \texorpdfstring{$p>q$}{p bigger q}}
\label{subsec: Inverse double phase problem for p bigger q}

In this section, we explain the necessary changes of Section \ref{subsec: Inverse double phase problem for p smaller q} to reconstruct the coefficient $a$ in the double phase problem \eqref{eq: PDE double phase} when $p>q$.

\begin{proof}[Proof of Theorem \ref{thm: main theorem double phase} for $p>q$]
    In a similar manner as in the proof of Theorem \ref{thm: main theorem double phase}, we consider the unique solution $u_\mu\in W^{1,p}(\Omega)$ to \eqref{eq: mu PDE} for some $p$-harmonic function $v\in C^{\infty}(\overline{\Omega})$ without critical points and large $\mu\gg 1$.

    First note that with the help of Proposition \ref{prop: asymptotic expansion p>q} we can repeat the argument of the proof of Theorem \ref{thm: main theorem double phase} to conclude that there holds
    \begin{equation}
    \label{eq: def I p>q}
        \begin{split}
            I(v,g)&\vcentcolon =\lim_{\mu\to\infty}\mu^{1-q}\left( \langle \Lambda_a\mu v, g\rangle- \langle\Lambda_0\mu v,g\rangle \right)\\
            &=\int_{\Omega} \left(A^p_v\nabla R_v +a|\nabla v|^{q-2}\nabla v\right)\cdot\nabla \omega\,dx.
        \end{split}
    \end{equation}
    Also the rest of the proof does not depend on the ordering of the exponents $p,q$ and hence we get the desired result by using the very same reasoning.
\end{proof}

	\medskip 
	
	\noindent\textbf{Acknowledgments.} 
C.~I.~C\^{a}rstea  was supported by the National Science and Technology Council (NSTC) grant number 113-2115-M-A49-018-MY3. P.~Zimmermann is supported by the Swiss National Science Foundation (SNSF), under the grant number 214500.

	\bibliography{refs} 
	
	\bibliographystyle{alpha}
	
\end{document}